\documentclass[a4paper,oneside,10pt]{article}
\usepackage[english]{babel}
\usepackage[T1]{fontenc}
\usepackage[a4paper,top=3cm,bottom=3cm,left=3cm,right=3cm,marginparwidth=1.75cm]{geometry}
\usepackage[colorlinks=true, linkcolor=blue, citecolor=teal, hypertexnames=false]{hyperref}
\usepackage{empheq}
\usepackage{graphicx}
\usepackage[colorinlistoftodos]{todonotes}
\usepackage{amsmath}
\usepackage{mathtools}
\usepackage{upgreek}
\usepackage{amssymb}
\usepackage{amsfonts}
\usepackage{amsthm}
\usepackage{hyperref}
\usepackage{enumitem}
\usepackage{mathrsfs}
\usepackage{fontenc}
\usepackage{verbatim}
\usepackage{framed}
\usepackage{algorithm}
\usepackage{bbm}
\usepackage{algpseudocode}
\usepackage{appendix}
\usepackage{color}
\usepackage{multirow}
\usepackage{lipsum}
\usepackage{authblk}
\usepackage{geometry}
\providecommand{\U}[1]{\protect\rule{.1in}{.1in}}

\newtheorem{theorem}{Theorem}[section]

\newtheorem{corollary}[theorem]{Corollary}

\newtheorem{assumption}[theorem]{Assumption}
\newtheorem{example}[theorem]{Example}

\newtheorem{lemma}[theorem]{Lemma}

\newtheorem{proposition}[theorem]{Proposition}
\newtheorem{remark}[theorem]{Remark}

\numberwithin{equation}{section}

\title{Global Convergence of Successive Approximations for Non-convex Stochastic Optimal Control Problems}

\author[1]{Shaolin Ji
\thanks{(jsl@sdu.edu.cn) Research was partially supported by the National Key R\&D Program of China (No.2023YFA1008701) and the Key Project of the National Natural Science Foundation of China (No.12431017).}
}

\author[2]{Rundong Xu
\thanks{(rdxu@tju.edu.cn) Research was supported by China Postdoctoral Science Foundation (No.2024M760481) and Shanghai Postdoctoral Excellence Program (No.2023201).}}

\affil[1]{\textit{Zhongtai Securities Institute for Financial Studies, Shandong University, Jinan 250100, China.}}

\affil[2]{\textit{Center for Applied Mathematics, Tianjin University, Tianjin 300072, China.}}

\date{}

\makeatletter
\newenvironment{breakablealgorithm}
  {
   \begin{center}
     \refstepcounter{algorithm}
     \hrule height.8pt depth0pt \kern2pt
     \renewcommand{\caption}[2][\relax]{
       {\raggedright\textbf{\ALG@name~\thealgorithm} ##2\par}%
       \ifx\relax##1\relax 
         \addcontentsline{loa}{algorithm}{\protect\numberline{\thealgorithm}##2}%
       \else 
         \addcontentsline{loa}{algorithm}{\protect\numberline{\thealgorithm}##1}%
       \fi
       \kern2pt\hrule\kern2pt
     }
  }{
     \kern2pt\hrule\relax
   \end{center}
  }
\makeatother

\begin{document}

\maketitle


\textbf{Abstract:} This paper focuses on finding approximate solutions to stochastic optimal control problems with control domains being not necessarily convex, 
where the state trajectory is subject to controlled stochastic differential equations.
The control-dependent diffusions make the traditional method of successive approximations (MSA) insufficient to reduce the value of cost functional in each iteration. 
Without adding extra terms over which to perform the Hamiltonian minimization, the MSA becomes sufficient by our novel error estimate involving a higher order backward adjoint equation.
Under certain convexity assumptions on the coefficients (no convexity assumptions on the control domains), the value of the cost functional descends to the global minimum as the number of iterations tends to infinity. 
In particular, a convergence rate is available for a class of generalized linear-quadratic systems.

\vspace{0.2cm}


\vspace{0.2cm}

\textbf{MSC2020 subject classification:} 93E20, 60H10, 60H30, 49M05

\addcontentsline{toc}{section}{\hspace*{1.8em}Abstract}

\section{Introduction}
In many numerical algorithms for solving optimal control problems (see \cite{Carmona-2021, Dong-2024, Eweinan-Hanjiequn-2017, Hu-Kazeykina-Ren-2019, BDL-Howard-2020, BDL-MSA-2020, Reisinger-Stockinger-Zhang-2023, Siska-Szpruch-SICON} and references therein), the method of successive approximations (MSA) is an efficient iterative scheme based on the Pontryagin optimality principle (also known as the maximum principle). 
This method involves successive integrations of the forward state equations and the backward adjoint equations, and updates the control variables by minimizing the Hamiltonian statically.
However, the systematic understanding of the MSA for stochastic optimal control problems remains incomplete, particularly when dealing with control-dependent diffusions and non-convex control domains, including simple cases like $\{0,1\}$. 
These ubiquitous settings can lead to two main issues. First, the traditional MSA algorithm, which involves finding a set with small measure to make an admissible control vary within a finite range, may not be sufficient to reduce the value of the cost functional in each algorithm iteration. 
Second, the reference maximum principle on which the modified MSA algorithms are based is often weaker than the usual one (see \cite{BDL-MSA-2020}, Theorem 2.4; \cite{Ji-Xu-MSA}, Lemma 2.7). 
This latter issue, in particular, can cause the cost functional to descend to a local minimum rather than a global minimum as the number of algorithm iterations increases.

To tackle these issues and intuitionally compare our main results with related literature,
we consider a stochastic optimal control problem with the standard cost functional
\begin{equation}
  \label{intro-cost-func}
  J(u(\cdot)) := \mathbb{E} \left[ \Phi(X^{u}(T)) + \int_{0}^{T} f(t, X^{u}(t), u(t)) \mathrm{d}t \right],
\end{equation}
being subject to a controlled stochastic differential equation (SDE)
\begin{equation}
  \label{intro-state-eq}
  \left\{
    \begin{array}
    [c]{rl}%
    \mathrm{d}X^{u}(t)= & b(t,X^{u}(t),u(t))\mathrm{d}t + \sigma(t,X^{u}(t),u(t)) \mathrm{d}W(t), \quad t \in [0,T],\\
    X^{u}(0) = & x_{0},
    \end{array}
  \right.
\end{equation}
with deterministic coefficients $b$, $\sigma$, $f$, $\Phi$ in suitable dimensions and a standard Brownian motion $W$ (maybe multi-dimensional) defined on a probability space $(\Omega, \mathcal{F}, \mathbb{P})$. 
(Throughout our work, we only consider the case that $\sigma$ depends on $u$.)
The objective is to minimize $J$ over $u(\cdot) \in \mathcal{U}[0,T]$, which is the collection of all admissible controls taking values in a nonempty set $U \subset \mathbb{R}^{k}$ (not necessarily convex).
Define the Hamiltonian as
\[
H(t,x,p,q,u) = p \cdot b(t,x,u) + q \cdot \sigma(t,x,u) + f(t,x,u).
\]
The adjoint equation of (\ref{intro-state-eq}) is then given by the backward stochastic differential equation (BSDE)
\begin{equation}
  \label{intro-1st-adj-eq}
  \left\{
    \begin{array}
    [c]{rl}%
    \mathrm{d}p^{u}(t)= & -H_{x}(t,X^{u}(t),p^{u}(t), q^{u}(t), u(t)) \mathrm{d}t + q^{u}(t) \mathrm{d}W(t), \quad t \in [0,T],\\
    p^{u}(T) = & \Phi_{x}(X^{u}(T)).
    \end{array}
  \right.
\end{equation}
Particularly, when $\sigma = 0$, (\ref{intro-1st-adj-eq}) degenerates into an ordinary differential equation and hence $q^{u}(t) \equiv 0$ for $t \in [0,T]$.
When $U$ is convex, the maximum principle stipulates that if $\overline{u}(\cdot)$ is an optimal control, and $(\overline{X}(\cdot), \overline{p}(\cdot), \overline{q}(\cdot))$
are the corresponding solutions to (\ref{intro-state-eq})-(\ref{intro-1st-adj-eq}) then we have
\begin{equation}
  \label{intro-H-MP}
  H(t,\overline{X}(t), \overline{p}(t), \overline{q}(t), \overline{u}(t)) = \inf\limits _{u \in U} H(t,\overline{X}(t), \overline{p}(t), \overline{q}(t), u)
\end{equation}
for almost all $(t,\omega) \in [0,T] \times \Omega$.
The underlying mechanism motivating each iteration in the MSA algorithm originates from the error estimate
\begin{equation}
  \label{intro-classical-error-est}
  J(u^{\prime}(\cdot)) - J(u(\cdot)) \leqslant
  \left\{
    \begin{array}[c]{lr}
       \int_{0}^{T} \Delta_{u} H(t) \mathrm{d}t + C \left( \int_{0}^{T} | u^{\prime}(t) - u(t) | \mathrm{d}t \right) ^{2}, & \sigma = 0, \\
       \mathbb{E} \left[\int_{0}^{T} \Delta_{u} H(t) \mathrm{d}t \right] + C \mathbb{E} \left[ \int_{0}^{T} \left\vert u^{\prime}(t) - u(t) \right\vert ^{2} \mathrm{d}t \right], & \sigma \neq 0,
    \end{array}
  \right.
\end{equation}
for any $u(\cdot), u^{\prime}(\cdot) \in \mathcal{U}[0,T]$, and some constant $C>0$ independent of $u(\cdot)$ and $u^{\prime}(\cdot)$, where
\[
\Delta_{u} H(t) = H(t, X^{u}(t), p^{u}(t), q^{u}(t) ,u^{\prime}(t)) - H(t, X^{u}(t), p^{u}(t), q^{u}(t) ,u(t)),
\]
provided that $b$, $\sigma$, $f$, $\Phi$, and their higher order derivatives satisfy appropriate assumptions (e.g., see \cite{BDL-MSA-2020,AAL-1982,Sethi-Siska-MSA}).
If $u(\cdot)$ violates (\ref{intro-H-MP}) and $v(t)$ is any admissible control found from the minimization condition
\begin{equation}
  \label{intro-H-mini}
  H(t,X^{u}(t),p^{u}(t),q^{u}(t),v) \longrightarrow \underset{v \in U}{\min}, \quad t \in [0,T],
\end{equation}
then, with $u^{\prime}(\cdot) = v(\cdot)$, the measure of the set $\{ t \in [0,T]: \mathbb{E} [\Delta_{u} H(t)] < 0 \}$ is strictly positive. For any $\tau$ in this set, take $u^{\prime}(\cdot)$ to be
$u_{\tau \varepsilon}(\cdot) = v(\cdot) 1_{[\tau - \varepsilon, \tau + \varepsilon]}(\cdot) + u(\cdot) 1_{[0,T] \setminus [\tau - \varepsilon, \tau + \varepsilon]}(\cdot)$ with $\varepsilon > 0$.
It follows from (\ref{intro-classical-error-est}) and the Lebesgue differentiation theorem that there is a constant $C_{u,\tau}>0$, depending only on $u(\cdot)$ and $\tau$, such that
\begin{equation}
  \label{intro-classical-J-reduce}
  J(u_{\tau \varepsilon}(\cdot)) - J(u(\cdot)) \leqslant
  \left\{
    \begin{array}[c]{lr}
      \underset{\leqslant -C_{u, \tau} \varepsilon < 0}{\underbrace{\int_{\tau - \varepsilon}^{\tau + \varepsilon} \Delta_{u} H(t) \mathrm{d}t}} + C \varepsilon^{2}, & \sigma = 0, \\
      \underset{\leqslant -C_{u, \tau} \varepsilon < 0}{\underbrace{\int_{\tau - \varepsilon}^{\tau + \varepsilon} \mathbb{E} \left[ \Delta_{u} H(t) \right] \mathrm{d}t}} + C \varepsilon, & \sigma \neq 0.
    \end{array}
  \right.
\end{equation}
The MSA for seeking numerical solutions to deterministic control systems, corresponding to $\sigma = 0$, was first proposed by Krylov et al. \cite{IA-Krylov1}. 
After that, many improved modifications of the MSA have been developed for a variety of deterministic control systems \cite{AAL2, IA-Krylov2, WeinanE-2018, AAL1} in the past few decades.
In this case, the remainder $C\varepsilon^{2}$ in (\ref{intro-classical-J-reduce}) is too negligible to prevent $J(u(\cdot))$ from decreasing when $\varepsilon$ is sufficiently small.

Nevertheless, for the stochastic case ($\sigma \neq 0$), it may fail to reduce $J(u(\cdot))$ by (\ref{intro-classical-J-reduce}) since the impact of the remainder $C\varepsilon$ cannot be ignored.
To overcome this problem, Kerimkulov et al. established a modified MSA in the pioneering work \cite{BDL-MSA-2020}, where they directly take $u^{\prime}(\cdot) = v^{\prime}(\cdot)$ in (\ref{intro-classical-error-est}) such that $v^{\prime}(t)$ is any admissible control found from the minimization of an augmented Hamiltonian
(Here, for the convenience of presentation, we adopt the version that is subsequently improved in \cite{Sethi-Siska-MSA}.)
\begin{equation}
  \label{intro-extended-H-mini}
  H(t,X^{u}(t),p^{u}(t),q^{u}(t),v) + \frac{\rho}{2} \left\vert v - u(t) \right\vert ^{2} \longrightarrow \underset{v \in U}{\min}, \quad t \in [0,T].
\end{equation}
Hence, for sufficiently large $\rho > 2C$, the second line in (\ref{intro-classical-error-est}) implies
\[
J(v^{\prime}(\cdot)) - J(u(\cdot)) \leqslant \left( 1 - \frac{2C}{\rho} \right) \mathbb{E} \left[\int_{0}^{T} \Delta_{u} H(t) \mathrm{d}t \right],
\]
which is strictly negative as long as $u(\cdot)$ violates (\ref{intro-H-MP}).
Based on this error estimate, their convergence result shows the approximate controls $\left\{ u^{m}(\cdot) \right\}_{m \in \mathbb{N}}$ with $(X^{m},p^{m},q^{m})$ produced by the MSA algorithm approximately satisfy (\ref{intro-extended-H-mini}) as $m \rightarrow \infty$. 
In addition, it is proved that if (\ref{intro-H-MP}) is a sufficient condition for optimality and if it has a structure that separates $(x,u)$-dependence of $b$, $\sigma$, and $f$, then the modified MSA converges at rate of $1/m$.
Inspired by this methodology, the second author of \cite{BDL-MSA-2020} and his collaborator recently made a remarkable progress \cite{Sethi-Siska-MSA} by proposing a gradient flow system. 
They demonstrated that the iterates of the corresponding modified MSA, when appropriately interpolated, converge to this gradient flow system at a rate of $\rho^{-1}$.
Moreover, under the convexity assumptions which ensure that (\ref{intro-H-MP}) is a sufficient condition, they proved that $J(u_{S}(\cdot)) \rightarrow \inf_{u(\cdot) \in \mathcal{U}[0,T]} J(u(\cdot))$ with rate $1/S$, where $S$ is the gradient flow time.
In our recent work \cite{Ji-Xu-MSA}, we extend such a modified MSA to recursive stochastic optimal control problems such that the cost functional is $J(\cdot) := Y^{u}(0)$, where $Y^{u}$ satisfies the BSDE
\[
  \left\{
    \begin{array}
    [c]{rl}%
    \mathrm{d}Y^{u}(t)= & -f(t,X^{u}(t),Y^{u}(t), Z^{u}(t), u(t)) \mathrm{d}t + q^{u}(t) \mathrm{d}W(t), \quad t \in [0,T],\\
    Y^{u}(T) = & \Phi(X^{u}(T)).
    \end{array}
  \right.
\]
The recursive type of $J$ is closely related to the recursive differential utility in economics \cite{Duffie-Epstein-1992} and degenerates into (\ref{intro-cost-func}) when $f$ is independent of $(y,z)$.

As (\ref{intro-extended-H-mini}) is weaker than (\ref{intro-H-mini}) with $\rho$ becoming larger,
adding the extra quadratic term in (\ref{intro-extended-H-mini}) results in $\left\{ u^{m}(\cdot) \right\}_{m \in \mathbb{N}}$ with $(X^{m},p^{m},q^{m})$ unlikely approximating (\ref{intro-H-MP}) as $m \rightarrow \infty$.
Consequently, the modified MSA algorithm may converge to a local minimum of $J$, for example, as Theorem 2.5 in \cite{BDL-MSA-2020} stated.
When $U$ is non-convex, it will be seen from Example \ref{sec-3-example} that the performance of modified MSA is sensitive to the choice of $\rho$.
Hence, to some extent, this modification seems not to be natural and it leads to a question whether it is possible to straightforwardly pursue a more efficient error estimate than (\ref{intro-classical-J-reduce}) for $\sigma \neq 0$,
ensuring that $J$ is decreasing with sufficiently small $\varepsilon$.

The major contribution of this paper provides a positive answer to the previous question by developing a new error estimate
\begin{align}
  \label{intro-stoch-nonconvex-J-reduce}
  J(u_{\tau \varepsilon}(\cdot)) - J(u(\cdot)) \leqslant & \int_{\tau - \varepsilon}^{\tau + \varepsilon} \mathbb{E} \left[ \mathcal{H}(t,X^{u}(t),p^{u}(t),q^{u}(t),P^{u}(t),v(t),u(t))  \right. \\
\nonumber  & - \left. \mathcal{H}(t,X^{u}(t),p^{u}(t),q^{u}(t),P^{u}(t),u(t),u(t))  \right] \mathrm{d}t + C \varepsilon^{\frac{3}{2}}
\end{align}
sharper than (\ref{intro-classical-J-reduce}) for $\sigma \neq 0$, where the $\mathcal{H}$-function is defined by (\ref{def-H-function}) and $P^{u}$ is the unique solution to the second-order adjoint equation (see (\ref{2nd-adj-eq}) in Section 2).
The introduction of $\mathcal{H}$ and $P^{u}$ are natural but essential to reducing $J$, due to their necessity in the expression of the general stochastic maximum principle (see \cite{Peng90}, Theorem 3; \cite{YongZhou}, Theorem 3.2)
and their significant role in lifting the remainder $C\varepsilon$ in (\ref{intro-classical-J-reduce}) up to $C\varepsilon^{3/2}$.
Employing (\ref{intro-stoch-nonconvex-J-reduce}), we establish a MSA algorithm and demonstrate that, for any $\delta>0$, there exists a positive integer $N_{\delta}$ such that if $m \geqslant N_{\delta}$ then $u^{m}(\cdot) $ minimizes the $\mathcal{H}$-function in terms of
\begin{equation}
\begin{array}
[c]{l}%
\mathbb{E}\left[  \int_{0}^{T}\mathcal{H}(t,X^{m}(t),p^{m}(t),q^{m}%
(t),P^{m}(t),u^{m}(t),u^{m}(t))\mathrm{d}t\right]  \\
\leqslant\inf\limits_{u(\cdot) \in \mathcal{U}[0,T]}\mathbb{E}\left[  \int_{0}^{T}\mathcal{H}(t,X^{m}(t),p^{m}(t),q^{m}(t),P^{m}(t),u(t),u^{m}(t))\mathrm{d}t\right] +\delta.
\end{array}
\label{intro-H-delta-minimize}
\end{equation} 
Furthermore, with some mild and convexity assumptions imposed on the coefficients (there is no convexity condition on $U$), (\ref{intro-H-delta-minimize}) is sufficient to obtain
\begin{equation}
\label{intro-near-optimal-ineq}
0 \leqslant J(u^{m}(\cdot))-J(\overline{u}(\cdot)) \leqslant C\delta^{\frac{1}{2}}
\end{equation}
with a positive constant $C$ independent of $\delta$ and $m$. Note that $J(\overline{u}(\cdot))$ can be replaced with $\inf_{u(\cdot) \in \mathcal{U}[0,T]}J(u(\cdot))$ if no optimal control exists. In this case, the $u^{m}(\cdot)$ satisfies (\ref{intro-near-optimal-ineq}) is called a $\delta^{\frac{1}{2}}$-optimal control that has been studied in \cite{Zhou-XY}.
Thus (\ref{intro-near-optimal-ineq}) also provides an approach to find a near-optimal control, and illustrates that near-optimal controls are indeed more available than exact optimal ones as mentioned in \cite{Zhou-XY}.
To meet the requirement of reducing the costs of computation in practice, obtaining (\ref{intro-near-optimal-ineq}) is not enough so
the convergence rate of $J(u^{m}(\cdot))$ descending to $J(\overline{u}(\cdot))$ is available as another contribution. 
Compared with the case $\sigma = 0$, where researchers obtained a $m^{-1}$-order convergence rate for the linear system with convex cost functionals in the Meyer type (e.g., see \cite{AAL1}, Theorem 2),
we obtain a $m^{-\frac{1}{2}}$-order convergence rate for a class of linear-quadratic systems as a special case of the stochastic control systems studied in this paper.

The outline of this paper is as follows. In section 2, we set up some preliminaries and formulate our problem. 
In section 3, we give a rigorous proof of (\ref{intro-stoch-nonconvex-J-reduce}) and use our MSA to find near-optimal controls and study the convergence rate.
A heuristic example (Example \ref{sec-3-example}) is given to illustrate our main result and to compare with the existing modified MSA.

\section{Preliminaries}

\subsection{Basic settings}
We start with fixing some general notation, vector spaces, and stochastic processes that will be used in the
sequel. The scalar product (resp. norm) of two real matrices $A$, $B$ is denoted by
$\left\langle A,B\right\rangle := \mathrm{tr}\{AB^{\intercal}\}$ (resp.
$\left\vert A\right\vert := \sqrt{\mathrm{tr}\left\{  AA^{\intercal}\right\}  }$), where the superscript $^{\intercal}$ denotes the transpose of vectors or
matrices. Denote by $I_{n}$ the $n \times n$ identity matrix.
Let $T \in (0,+\infty)$ be a finite time horizon, $\mathrm{Leb}_{[0,T]}$ be the Lebesgue measure on $[0,T]$,
and $(\Omega,\mathcal{F},\mathbb{P})$ be a complete probability space
on which a standard $d$-dimensional Brownian motion $\{ W=(W^{1}(t),W^{2}(t),\ldots, W^{d}(t))^{\intercal} : t \in [0,T] \}$ is defined.
Denote $\mathbb{F} := \left\{  \mathcal{F}_{t}: t \in [0,T] \right\}$ the $\mathbb{P}$-augmentation of the natural filtration generated by $W$.
Particularly in this paper, for any measurable and bounded function $f$ defined on (or on any subset of) any Euclidean space, we simply denote the bound of $f$ by $\sup |f|$.
For any $p,q \geqslant 1$, $n \in \mathbb{N}_{+}$, we introduce the following spaces and notation.

\begin{itemize}
  \item $L_{\mathcal{F}_{T}}^{p}(\Omega;\mathbb{R}^{n})$: the space of $\mathcal{F}_{T}$-measurable, $\mathbb{R}^{n}$-valued random variables $\xi$ such that
$$
\left\Vert \xi \right\Vert _{L^{p}} := \left(\mathbb{E}\left[  \left\vert \xi\right\vert ^{p}\right]\right)^{\frac{1}{p}}  < +\infty.
$$

\item $\mathcal{M}_{\mathbb{F}}^{p,q}([0,T];\mathbb{R}^{n})$: the space of
$\mathbb{F}$-adapted, $\mathbb{R}^{n}$-valued processes $\varphi(\cdot)$ on
$[0,T]$ such that
\[
\left\Vert \varphi(\cdot)\right\Vert _{\mathcal{M}^{p,q}}:= \left\Vert \left( \int_{0}^{T}\left\vert \varphi(t) \right\vert ^{p} \mathrm{d}t \right)  ^{\frac{1}{p}} \right\Vert _{L^{q}} < +\infty.
\]
In particular, we denote by $\mathcal{M}_{\mathbb{F}}^{p}([0,T];\mathbb{R}^{n})$ the above space when $p=q$.

\item $\mathcal{S}_{\mathbb{F}}^{p}([0,T];\mathbb{R}^{n})$: the space of continuous
processes $\varphi(\cdot)\in$ $\mathcal{M}_{\mathbb{F}}^{p}([0,T];\mathbb{R}%
^{n})$ such that
\[
\left\Vert \varphi \right\Vert _{\mathcal{S}^{p}} := \left\Vert \sup\limits_{t\in[0,T]}\left\vert \varphi(t)\right\vert \right\Vert _{L^{p}} < +\infty.
\]
\end{itemize}

\subsection{Problem Formulation}
Let $U$ be a nonempty, compact subset of $\mathbb{R}^{k}$.
The objective is to minimize
\begin{equation}
J(u(\cdot)):=\mathbb{E}\left[ \Phi(X^{u}(T))+\int_{0}^{T} f(t,X^{u}(t),u(t))\mathrm{d}t \right]
\label{cost-func}%
\end{equation}
over the set of $U$-valued admissible controls
\begin{equation}
\mathcal{U}[0,T]:=\{u(\cdot): [0,T] \times \Omega \mapsto U |\ u(\cdot) \text{ is } \mathbb{F} \text{-progressively measurable.} \},
\label{control-integrable}%
\end{equation}
where the state trajectory $X^{u}(\cdot)$ is subject to a controlled stochastic differential equation
\begin{equation}
\left\{
\begin{array}
[c]{rl}%
dX^{u}(t)= & b(t,X^{u}(t),u(t))\mathrm{d}t+\sigma(t,X^{u}(t),u(t))\mathrm{d}W(t),\\
X^{u}(0)= & x_{0},
\end{array}
\right.  \label{state-eq}%
\end{equation}
with a initial value $x_{0} \in\mathbb{R}^{n}$, and $b:[0,T]\times
\mathbb{R}^{n}\times U\longmapsto\mathbb{R}^{n}$, $\sigma:[0,T]\times
\mathbb{R}^{n}\times U\longmapsto\mathbb{R}^{n\times d}$, $f:[0,T]\times
\mathbb{R}^{n}\times U\longmapsto
\mathbb{R}$, and $\Phi:\mathbb{R}^{n}\longmapsto\mathbb{R}$ being deterministic, measurable functions.
Unless indicated, we always assume that $\sigma$ depends on $u$ throughout this paper.
Suppose that there exists at least one $u(\cdot) \in \mathcal{U}[0,T]$ minimizing (\ref{cost-func}) over $\mathcal{U}[0,T]$.
We impose the following assumptions on the coefficients of (\ref{state-eq}).
\begin{assumption}
\label{assum-1}
(i) $b$, $\sigma$, $f$, $\Phi$ are twice continuously differentiable with respect to $x$; 
$b$, $\sigma$, $f$, $b_{x}$, $\sigma_{x}$, $f_{x}$, $b_{xx}$, $\sigma_{xx}$, $f_{xx}$, are jointly continuous in $\left( t,x \right)$.

(ii) $b_{x}$, $\sigma_{x}$, $b_{xx}$, $\sigma_{xx}$, $f_{xx}$, $\Phi_{xx}$ are bounded.

(iii) There exists a constant $L>0$ such that, for any $(t,x,u) \in [0,T] \times \mathbb{R}^{n} \times U$,
\begin{align*}
  \left\vert b(t,x,u) \right\vert + \left\vert \sigma(t,x,u) \right\vert \leqslant & \ L\left(  1 + \left\vert x\right\vert +\left\vert u\right\vert \right) ; \\
   \left\vert f(t,x,u) \right\vert \leqslant L\left(  1 + \left\vert x \right\vert ^{2} +\left\vert u \right\vert ^{2} \right); \ & \ \left\vert \Phi(x) \right\vert \leqslant L\left(  1 + \left\vert x \right\vert ^{2} \right).
\end{align*}

(iv) $b_{xx}$, $\sigma_{xx}$, $f_{xx}$, $\Phi_{xx}$ are $L$-Lipschitz continuous in $x$, uniformly in $(t,u)$, i.e.
\[
\left\vert \psi_{xx}(t,x_{1},u)-\psi_{xx}(t,x_{2},u)\right\vert \leqslant L\left\vert x_{1}-x_{2}\right\vert ,\text{ }\forall (t,u)\in\lbrack0,T]\times U,\text{ }x_{1},x_{2}\in\mathbb{R}^{n},
\]
where $\psi=b,\sigma,f,\Phi$.
\end{assumption}

\begin{remark}
Since $U$ is compact, it can be deduced from the above assumption that
$b$, $\sigma$, $f_{x}$ are bounded by $\tilde{L}\left(  1+\left\vert x\right\vert \right)  $ for some $\tilde{L}>0$ and $f(t,0,u)$ is uniformly bounded for all $(t,u) \in [0,T] \times U$.
\end{remark}

Put $\alpha=\sup\left\{\left\vert u \right\vert : u\in U \right\}$.
The following lemma provides a well-posedness result and a $L^{p}$-estimate of the controlled SDE (\ref{state-eq}).

\begin{lemma}
\label{lem-est-state-eq} Let Assumption \ref{assum-1} hold. Then, for any given
$u(\cdot)\in\mathcal{U}[0,T]$, (\ref{state-eq}) admits a unique strong solution $X^{u}(\cdot)$.
Moreover, $X^{u}(\cdot)$ is bounded in $\mathcal{S}_{\mathcal{F}
}^{8}([0,T];\mathbb{R}^{n})$ uniformly across all $u(\cdot)\in\mathcal{U}[0,T]$, i.e.
\begin{equation}
\label{Xu-uniform-bounded}
\sup_{u(\cdot)\in\mathcal{U}[0,T]}\mathbb{E}\left[  \sup\limits_{t\in
\lbrack0,T]}\left\vert X^{u}(t)\right\vert ^{8}\right]  \leqslant C,
\end{equation}
where $C$ depends only on $n$, $d$, $T$, $\alpha$, $L$, $x_{0}$, $\sup | b_{x} |$, $\sup | \sigma_{x} |$.
\end{lemma}

\begin{proof}
For any $u(\cdot) \in \mathcal{U}[0,T]$, under Assumption \ref{assum-1},
there exists a $\mathbb{F}$-progressively measurable process $X^{u}(\cdot)$ being unique strong solution to the SDE in (\ref{state-eq}) (\cite{Carmona}, Theorem 1.2; \cite{YongZhou}, Chapter I, Theorem 6.3), which satisfies the following standard estimate:
\begin{equation}
\left\Vert X^{u} \right\Vert _{\mathcal{S}^{8}}^{8}
\leqslant C \left(  \left\vert x_{0}\right\vert ^{8} + \left\Vert b(\cdot,0,u(\cdot)) \right\Vert _{\mathcal{M}^{1,8}}^{8} + \left\Vert \sigma(\cdot,0,u(\cdot)) \right\Vert _{\mathcal{M}^{2,8}}^{8} \right) .
\label{est-forward-state-eq}
\end{equation}
Then, under Assumption \ref{assum-1}, (\ref{Xu-uniform-bounded}) follows from (\ref{est-forward-state-eq}) and the boundedness of $U$.
\end{proof}

For the convenience of introducing the adjoint equations of (\ref{state-eq}), we set
\[
\begin{array}
[c]{l}%
b(\cdot)=\left(  b^{1}(\cdot),b^{2}(\cdot),\ldots,b_{{}}^{n}(\cdot)\right)
^{\intercal}\in\mathbb{R}^{n},\\
\sigma(\cdot)=\left(  \sigma^{1}(\cdot),\sigma^{2}(\cdot),\ldots,\sigma
^{d}(\cdot)\right)  \in\mathbb{R}^{n\times d},\\
\sigma^{i}(\cdot)=\left(  \sigma^{1i}(\cdot),\sigma^{2i}(\cdot),\ldots
,\sigma_{{}}^{ni}(\cdot)\right)  ^{\intercal}\in\mathbb{R}^{n},i=1,2,\ldots,d.
\end{array}
\]
In addition, for $\psi=b$, $\sigma$, $f$, we simply denote
\begin{equation}%
\begin{array}
[c]{lll}%
\psi^{u}(t)=\psi(t,X^{u}(t),u(t)), & \psi_{x}^{u}(t)=\psi_{x}(t,X^{u}(t),u(t)), & \psi_{xx}^{u}(t)=\psi_{xx}(t,X^{u}(t),u(t)),
\end{array}
\label{def-notation-1}%
\end{equation}
and denote 
\[
\begin{array}
[c]{ll}
\sigma_{x}^{u,i}(t)=\sigma_{x}^{i}(t,X^{u}(t),u(t)), & \sigma_{xx}^{u,i}(t)=\sigma_{xx}^{i}(t,X^{u}(t),u(t)), \ i=1,\ldots,d.
\end{array}
\]
The first-order adjoint equation is given by
\begin{equation}
\left\{
\begin{array}
[c]{rl}%
dp^{u}(t)= & -\left[  \left(  b_{x}^{u}(t)\right)  ^{\intercal}p^{u}(t) +\sum\limits_{i=1}^{d}\left(  \sigma_{x}^{u,i}(t)\right)  ^{\intercal}q^{u,i}(t)+f_{x}^{u}(t)\right]  \mathrm{d}t \\
& +\sum\limits_{i=1}^{d}q^{u,i}(t)\mathrm{d}W_{i}(t), \quad  t \in [0,T], \\
p^{u}(T)= & \Phi_{x}(X^{u}(T)),
\end{array}
\right.  
\label{1st-adj-eq}%
\end{equation}
and the second-order adjoint equation is given by
\begin{equation}
\left\{
\begin{array}
[c]{rl}%
dP^{u}(t)= & -\left\{  \left(  b_{x}^{u}(t)\right)  ^{\intercal}%
P^{u}(t)+\left(  P^{u}(t)\right)  ^{\intercal}b_{x}^{u}(t)+\sum\limits_{i=1}%
^{d}\left(  \sigma_{x}^{u,i}(t)\right)  ^{\intercal}P^{u}(t)\sigma_{x}%
^{u,i}(t)\right.  \\
& +\sum\limits_{i=1}^{d}\left[  \left(  \sigma_{x}^{u,i}(t)\right)
^{\intercal}Q^{u,i}(t)+\left(  Q^{u,i}(t)\right)  ^{\intercal}\sigma_{x}%
^{u,i}(t)\right]  \\
& +\left.  H_{xx}(t,X^{u}(t),p^{u}(t),q^{u}(t),u(t))\right\}  \mathrm{d}t+\sum
\limits_{i=1}^{d}Q^{u,i}(t)\mathrm{d}W_{i}(t),\text{ \ }t\in\lbrack0,T],\\
P^{u}(T)= & \Phi_{xx}(X^{u}(T)).
\end{array}
\right.  
\label{2nd-adj-eq}%
\end{equation}
By a solution to (\ref{1st-adj-eq}) we mean a multiple $(p^{u}(\cdot), q^{u,1}(\cdot), \ldots ,q^{u,d}(\cdot))$ 
of $\mathbb{F}$-progressively measurable processes with values in $\mathbb{R}^{n} \times \mathbb{R}^{n \times d}$ satisfying (\ref{1st-adj-eq}).
The definition of solutions to (\ref{2nd-adj-eq}) is similar so we will not repeat it below.
The Hamiltonian $H$ is defined by
\begin{align}%
\label{def-Hamiltonian}
H(t,x,p,q,u) = & \ p^{\intercal}b(t,x,u) + \left\langle q,\sigma(t,x,u)\right\rangle +f(t,x,u),\\
\nonumber & (t,x,p,q,u) \in [0,T]\times\mathbb{R}^{n}\times\mathbb{R}^{n} \times \mathbb{R}^{n\times d} \times U.
\end{align}


Under Assumption \ref{assum-1}, applying Theorem 5.1 in \cite{Karoui97} yields the well-posedness of (\ref{1st-adj-eq}).

\begin{lemma}
\label{adj-1st-lem} Let Assumption \ref{assum-1} hold. Then, for any
$u(\cdot)\in\mathcal{U}[0,T]$, (\ref{1st-adj-eq}) admits a unique
solution $\left(  p^{u}(\cdot),q^{u}(\cdot)\right)  \in\mathcal{S}_{\mathbb{F}}^{8}([0,T];\mathbb{R}^{n})\times \left(\mathcal{M}_{\mathbb{F}}^{2,8}([0,T];\mathbb{R}^{n})\right)^{d}$,
where $q^{u}(\cdot)=\left(  q^{u,1}(\cdot),\ldots,q^{u,d}(\cdot)\right)  $.
\end{lemma}


Using the above lemma and Theorem 5.1 in \cite{Karoui97} yields the well-posedness of (\ref{2nd-adj-eq}).
\begin{lemma}
\label{adj-2nd-lem} Let Assumption \ref{assum-1} hold. Then, for any
$u(\cdot)\in\mathcal{U}[0,T]$, (\ref{2nd-adj-eq}) admits a unique
solution $\left(  P^{u}(\cdot),Q^{u}(\cdot)\right)  \in \mathcal{S}_{\mathbb{F}}^{4}([0,T];\mathbb{R}^{n}) \times \left(  \mathcal{M}_{\mathbb{F}}^{2,4}([0,T];\mathbb{S}^{n\times n})\right)  ^{d}$, where $Q^{u}(\cdot)=\left(  Q^{u,1}(\cdot),\ldots,Q^{u,d}(\cdot)\right)  $. 
\end{lemma}

Define $\mathcal{H}:[0,T] \times \mathbb{R}^{n} \times \mathbb{R}^{n} \times \mathbb{R}^{n\times d}\times\mathbb{S}^{n\times n}\times U \times U\longmapsto\mathbb{R}$ by
\begin{equation}%
\begin{array}
[c]{rl}
& \mathcal{H}(t,x,p,q,P,v,u)\\
= & H(t,x,p,q,v)+\dfrac{1}{2}\sum\limits_{i=1}^{d}\left(  \sigma
^{i}(t,x,v)-\sigma^{i}(t,x,u)\right)  ^{\intercal}P\left(  \sigma
^{i}(t,x,v)-\sigma^{i}(t,x,u)\right)  \\
& -\dfrac{1}{2}\sum\limits_{i=1}^{d}\left(  \sigma^{i}(t,x,u)\right)
^{\intercal}P\left(  \sigma^{i}(t,x,u)\right)  .
\end{array}
\label{def-H-function}%
\end{equation}
Then we can rewrite the following stochastic maximum principle (\cite{Peng90}, Theorem 3; \cite{YongZhou}, Theorem 3.2) for 
stochastic optimal control problem (\ref{cost-func})-(\ref{state-eq}) by $\mathcal{H}$-function.

\begin{theorem}
\label{thm-SMP}
Suppose Assumption \ref{assum-1}. Let $\left(  \overline{X}(\cdot),\overline{u}(\cdot)\right)$ be the optimal pair. 
Then, for $\mathrm{Leb}_{[0,T]} \otimes \mathbb{P}$-a.e. $(t,\omega)$, we have
\begin{equation}
\mathcal{H}(t,\overline{X}(t),\overline{p}(t),\overline{q}(t),\overline{P}(t),v,\overline{u}(t))\geqslant\mathcal{H}(t,\overline{X}(t),\overline{p}(t),\overline{q}(t),\overline{P}(t),\overline{u}(t),\overline{u}(t)), \quad \forall v\in U,
\label{SMP}%
\end{equation}
where $\left(  \overline{p}(\cdot),\overline{q}(\cdot)\right)  $, $\overline{P}(\cdot)$ are the solutions to
(\ref{1st-adj-eq}), (\ref{2nd-adj-eq}) respectively, corresponding to $\overline{u}(\cdot)$.
\end{theorem}

\begin{remark}
It should be highlighted that $\mathcal{H}$-function plays the same role as the common one in the global stochastic maximum principle (see \cite{YongZhou}, Theorem 3.2) 
but is slightly generalized to rely on double control arguments, for the demand of iterations in our MSA algorithm. 
\end{remark}

Based on Theorem \ref{thm-SMP}, for any $u(\cdot) \in \mathcal{U}[0,T]$ and each $(t,\omega) \in [0,T] \times \Omega$, let $v(\cdot)$ be the process found from
\begin{equation}
v(t,\omega) \in \underset{v\in U}{\arg\min} \mathcal{H}(t,X^{u}(t,\omega),p^{u}(t,\omega),q^{u}(t,\omega),P^{u}(t,\omega),v,u(t,\omega)).
\label{def-control-argmin-H}%
\end{equation}
Since the mapping $v \mapsto \mathcal{H}(t,x,p,q,P,v,u)$ is continuous and $U$ is compact, Proposition D.5 in
\cite{OHL-JL-1996} guarantees the existence of an appropriate measurable selection $V(t,x,p,q,P,u)$, which minimizes
$\mathcal{H}(t,x,p,q,P,\cdot,u)$ over $U$ for each $(t,x,p,q,P,u)$. Setting $v(t)=V(t,X^{u}(t),p^{u}(t),q^{u}(t),P^{u}(t),u(t))$,
it is easy to verify $v(\cdot)\in\mathcal{U}[0,T]$ as the $\mathbb{F}$-progressive measurability of $v(\cdot)$ can be deduced by Doob's measurability theorem.
Then, define
\begin{align}
\label{def-Delta-H}
\Delta_{u}\mathcal{H}(t) = & \mathcal{H}(t,X^{u}(t),p^{u}(t),q^{u}(t),P^{u}(t),v(t),u(t)) \\
\nonumber & - \mathcal{H}(t,X^{u}(t),p^{u}(t),q^{u}(t),P^{u}(t),u(t),u(t)), \ t \in [0,T],
\end{align}
\begin{equation}
\label{def-mu-control}
\mu(u(\cdot)) = \mathbb{E}\left[ \int_{0}^{T} \Delta_{u} \mathcal{H}(t) \mathrm{d}t \right].
\end{equation}
In view of (\ref{def-control-argmin-H}), $\Delta_{u} \mathcal{H}(t) \leqslant 0$, $\mu(u(\cdot)) \leqslant 0$. Moreover, if $\mu(u(\cdot)) = 0$, then it means that
$u(\cdot)$ satisfies (\ref{SMP}). Hence we can regard $\mu(u(\cdot))$ as characterizing the extent to which the admissible control  $u(\cdot)$ deviates from satisfying the necessary conditions for optimality.

\section{The algorithm}
In this section, we establish an algorithm and then apply it to finding the near-optimal controls to (\ref{cost-func})-(\ref{state-eq}). Particularly, we obtain the convergence rate of the algorithm where there are some additional assumptions imposed on the coefficients in (\ref{state-eq}).

\subsection{An error estimate to construct the algorithm}
Given $\tau \in [0,T]$, $\varepsilon>0$ arbitrarily, put $E_{\tau \varepsilon}=[\tau - \varepsilon, \tau + \varepsilon] \bigcap [0,T]$.
For any $u(\cdot) \in \mathcal{U}[0,T]$, consider the two-parameter family of admissible controls
\begin{equation}
u_{\tau\varepsilon}(t)=\left\{
\begin{array}
[c]{ll}%
v(t), & t\in E_{\tau\varepsilon}\\
u(t), & t\in\lbrack0,T]\backslash E_{\tau\varepsilon},
\end{array}
\right.  \label{def-spike-variation}%
\end{equation}
where $v(\cdot)$ is determined by (\ref{def-control-argmin-H}).

\begin{proposition}
\label{prop-costfunc-dominate} Let Assumption \ref{assum-1} hold. Then
there exists a universal constant $C>0$ (independent of $u(\cdot)$, $\tau$, $\varepsilon$) such that
\begin{equation}
J\left(  u_{\tau \varepsilon}(\cdot)\right)  -J\left(  u(\cdot)\right)
\leqslant\mathbb{E}\left[  \int_{E_{\tau \varepsilon}}\Delta_{u}\mathcal{H}(t)\mathrm{d}t\right]  +C\varepsilon^{\frac{3}{2}}
\label{est-Jv-Ju}
\end{equation}
for any $u(\cdot) \in \mathcal{U}[0,T]$ and
$u_{\tau\varepsilon}(\cdot)$ defined by (\ref{def-spike-variation}),
where $\Delta_{u} \mathcal{H}$ is defined by (\ref{def-Delta-H}).
\end{proposition}

To prove Proposition \ref{prop-costfunc-dominate}, we introduce the following two SDEs:
\begin{equation}
\left\{
\begin{array}
[c]{rl}%
dX_{1}(t)= & b_{x}(t)X_{1}(t)\mathrm{d}t+\sum\limits_{i=1}^{d}\left[  \sigma_{x}%
^{i}(t)X_{1}(t)+\hat{\sigma}^{i}(t)\mathrm{1}_{E_{\tau\varepsilon}}(t)\right]
\mathrm{d}W_{i}(t),\\
X_{1}(0)= & 0,
\end{array}
\right.  \label{var-eq-1}%
\end{equation}%
\begin{equation}
\left\{
\begin{array}
[c]{rl}%
dX_{2}(t)= & \left[  b_{x}(t)X_{2}(t)+\hat{b}(t)\mathrm{1}_{E_{\tau
\varepsilon}}(t)+\frac{1}{2}b_{xx}(t)X_{1}(t)X_{1}(t)\right]  \mathrm{d}t\\
& +\sum\limits_{i=1}^{d}\left[  \sigma_{x}^{i}(t)X_{2}(t)+\frac{1}{2}%
\sigma_{xx}^{i}(t)X_{1}(t)X_{1}(t)+\hat{\sigma}_{x}^{i}(t)X_{1}(t)\mathrm{1}%
_{E_{\tau\varepsilon}}(t)\right]  \mathrm{d}W_{i}(t),\\
X_{2}(0)= & 0,
\end{array}
\right.  \label{var-eq-2}%
\end{equation}
where, for $\psi=b$, $\sigma^{1}, \ldots, \sigma^{d} $,
\begin{align*}
  \hat{\psi}(t):= & \ \psi(t,X^{u}(t),v(t))-\psi(t,X^{u}(t),u(t)),\\
\psi_{xx}(t)X_{1}(t)X_{1}(t):= & \left(  \mathrm{tr}\left\{  \psi_{xx}%
^{1}(t)X_{1}(t)X_{1}^{\intercal}(t)\right\}  ,\ldots,\mathrm{tr}\left\{
\psi_{xx}^{n}(t)X_{1}(t)X_{1}^{\intercal}(t)\right\}  \right)  ^{\intercal}.
\end{align*}
In the upcoming proof, the universal constant $C$ may depend on $n$, $d$, $T$, $\alpha$,
$L$, $\sup | b_{x} |$, $\sup | \sigma_{x} |$, $\sup | \Phi_{xx} |$, $\sup | f_{xx} |$, and will
change from line to line in our proof.

\begin{proof}[Proof of Proposition \ref{prop-costfunc-dominate}]

At first, we need to prove an estimate
\begin{equation}
\mathbb{E}\left[  \sup\limits_{t\in\lbrack0,T]}\left\vert X_{\tau\varepsilon
}^{u}(t)-X^{u}(t)-X_{1}(t)-X_{2}(t)\right\vert ^{2}\right]  \leqslant C\varepsilon^{3},
\label{error-X-higher-order}%
\end{equation}
where $X_{\tau\varepsilon}^{u}(\cdot)$ is the state trajectory corresponding to $u_{\tau\varepsilon}(\cdot)$,
$X_{1}(\cdot)$, $X_{2}(\cdot)$ are solutions to (\ref{var-eq-1}), (\ref{var-eq-2}) respectively,
and $C>0$ is independent of $u(\cdot)$, $v(\cdot)$ and $\varepsilon$.
To this end, from (\ref{est-forward-state-eq}) and (\ref{def-spike-variation}), one can verify by a standard estimate for SDEs (\cite{Carmona, YongZhou}) that
\begin{equation}
\mathbb{E}\left[  \sup\limits_{t\in\lbrack0,T]}\left\vert X_{1}(t)\right\vert
^{8}+\sup\limits_{t\in\lbrack0,T]}\left\vert X_{2}(t)\right\vert ^{4}\right]
\leqslant C\varepsilon^{4},\label{est-X1-X2}%
\end{equation}
where $C>0$ is independent of $u(\cdot)$, $v(\cdot)$, $\tau$ and $\varepsilon$.
Set $\hat{X}(\cdot)=X_{1}(\cdot)+X_{2}(\cdot)$,
\begin{align*}
\tilde{b}_{xx}(s) = & \int_{0}^{1} \int_{0}^{1} \theta b_{xx}(s,\overline{X}(s) + \rho\theta\hat{X}(s),u_{\tau\varepsilon}(s))d\rho d\theta,\\
\tilde{\sigma}_{xx}^{i}(s) = & \int_{0}^{1} \int_{0}^{1} \theta \sigma_{xx}^{i}(s,\overline{X}(s) + \rho \theta \hat{X}(s),u_{\tau\varepsilon}(s)) d\rho d\theta, \ i=1,\ldots,d.
\end{align*}
We have
\begin{align*}
  & \int_{0}^{t}b(s,X^{u}(s)+\hat{X}(s),u_{\tau\varepsilon}(s))\mathrm{d}s+\sum\limits_{i=1}^{d}\int_{0}^{t}\sigma^{i}(s,X^{u}(s)+\hat{X}(s),u_{\tau \varepsilon}(s))\mathrm{d}W_{i}(s)\\
\nonumber = & \int_{0}^{t}\left[  b(s,X^{u}(s),u_{\tau\varepsilon}(s))+b_{x}(s,X^{u}(s),u_{\tau\varepsilon}(s))\hat{X}(s)+\tilde{b}_{xx}(s)\hat{X}(s)\hat{X}(s)\right] \mathrm{d}s\\
\nonumber & +\sum\limits_{i=1}^{d}\int_{0}^{t}\left[  \sigma^{i}(s,X^{u}(s),u_{\tau \varepsilon}(s))+\sigma_{x}^{i}(s,X^{u}(s),u_{\tau\varepsilon}(s))\hat{X}(s)+\tilde{\sigma}_{xx}^{i}(s)\hat{X}(s)\hat{X}(s)\right]  \mathrm{d}W_{i}(s)\\
\nonumber = & \int_{0}^{t}\left[  b(s)+b_{x}(s)\hat{X}(s)+\frac{1}{2}b_{xx}(s)\hat{X}(s)\hat{X}(s)\right] \mathrm{d}s\\
\nonumber & +\sum\limits_{i=1}^{d}\left[  \int_{0}^{t}\sigma^{i}(s)+\sigma_{x}^{i}(s)\hat{X}(s)+\frac{1}{2}\sigma_{xx}^{i}(s)\hat{X}(s)\hat{X}(s)\right] \mathrm{d}W_{i}(s)\\
\nonumber & +\int_{0}^{t}\hat{b}(s)\mathrm{1}_{E_{\tau\varepsilon}}(s)\mathrm{d}s+\int_{0}^{t}\hat{b}_{x}(s)\hat{X}(s)\mathrm{1}_{E_{\tau\varepsilon}}(s)\mathrm{d}s+\int_{0}^{t}\left[  \tilde{b}_{xx}(s)-b_{xx}(s)\right]  \hat{X}(s)\hat{X}(s)\mathrm{d}s\\
\nonumber & +\sum\limits_{i=1}^{d}\int_{0}^{t}\hat{\sigma}^{i}(s)\mathrm{1}_{E_{\tau\varepsilon}}(s)\mathrm{d}W_{i}(s)+\sum\limits_{i=1}^{d}\int_{0}^{t}\hat{\sigma}_{x}^{i}(s)\hat{X}(s)\mathrm{1}_{E_{\tau\varepsilon}}(s)\mathrm{d}W_{i}(s)\\
\nonumber & +\sum\limits_{i=1}^{d}\int_{0}^{t}\left[  \tilde{\sigma}_{xx}^{i}(s)-\sigma_{xx}^{i}(s)\right]  \hat{X}(s)\hat{X}(s)\mathrm{d}W_{i}(s)\\
\nonumber = & \ X^{u}(t) - x_{0} + \hat{X}(t) + \int_{0}^{t}\Pi_{\tau\varepsilon}(s)\mathrm{d}s + \sum\limits_{i=1}^{d}\int_{0}^{t}\Lambda_{\tau\varepsilon}^{i}(s)\mathrm{d}W_{i}(s),
\end{align*}
where (using (\ref{var-eq-1}) and (\ref{var-eq-2}))
\begin{align*}
\Pi_{\tau\varepsilon}(s) = & \ \hat{b}_{x}(s)X_{2}(s)\mathrm{1}_{E_{\tau
\varepsilon}}(s)+\frac{1}{2}b_{xx}(s)\left[  \hat{X}(s)\hat{X}(s)-X_{1}%
(t)X_{1}(t)\right]  \\
& +\left[  \tilde{b}_{xx}(s)-b_{xx}(s)\right]  \hat{X}(s)\hat{X}(s), \\
\Lambda_{\tau\varepsilon}^{i}(s)= & \ \hat{\sigma}^{i}(s)X_{2}(s)\mathrm{1}_{E_{\tau\varepsilon}}(s)+\frac{1}{2}\sigma_{xx}^{i}(s)\left[  \hat{X}%
(s)\hat{X}(s)-X_{1}(t)X_{1}(t)\right]  \\
& +\left[  \tilde{\sigma}_{xx}^{i}(s)-\sigma_{xx}^{i}(s)\right]  \hat{X}(s)\hat{X}(s).
\end{align*}
Thus, we obtain
\begin{align*}
X^{u}(t) + \hat{X}(t) = & \ x_{0} + \int_{0}^{t}b(s,X^{u}(s) \hat{X}(s),u_{\tau \varepsilon}(s))\mathrm{d}s \\
& + \sum\limits_{i=1}^{d}\int_{0}^{t}\sigma^{i}(s,X^{u}(s) + \hat{X}(s),u_{\tau\varepsilon}(s))\mathrm{d}W_{i}(s) \\
& -\int_{0}^{t} \Pi_{\tau\varepsilon}(s)\mathrm{d}s - \sum\limits_{i=1}^{d}\int_{0}^{t}\Lambda_{\tau\varepsilon}^{i}(s)\mathrm{d}W_{i}(s).
\end{align*}
Since
\[
X_{\tau\varepsilon}^{u}(t)=x_{0}+\int_{0}^{t}b(s,X_{\tau\varepsilon}%
^{u}(s),u_{\tau\varepsilon}(s))\mathrm{d}s+\sum\limits_{i=1}^{d}\int_{0}^{t}\sigma
^{i}(s,X_{\tau\varepsilon}^{u}(s),u_{\tau\varepsilon}(s))\mathrm{d}W_{i}(s),
\]
we can derive%
\begin{align}
\label{eq-higher-order}
\left(  X_{\tau\varepsilon}^{u}-X^{u}-\hat{X}\right)  (t)= & \int_{0}^{t}A_{\tau\varepsilon}(s)\left(  X_{\tau\varepsilon}^{u}-X^{u}-\hat{X}\right)  (s)\mathrm{d}s \\
\nonumber & + \sum\limits_{i=1}^{d}\int_{0}^{t}B_{\tau\varepsilon}^{i}(s)\left( X_{\tau\varepsilon}^{u} - X^{u}-\hat{X}\right) (s)\mathrm{d}W_{i}(s) \\
\nonumber & +\int_{0}^{t}\Pi_{\tau\varepsilon}(s)\mathrm{d}s+\sum\limits_{i=1}^{d}\int_{0}^{t}\Lambda_{\tau\varepsilon}^{i}(s)\mathrm{d}W_{i}(s),
\end{align}
where
\begin{align*}
  A_{\tau\varepsilon}(s)= & \int_{0}^{1}b_{x}\left(  s,\overline{X}(s)+\hat
{X}(s)+\theta\left(  X_{\tau\varepsilon}^{u}(s)-X^{u}(s)-\hat{X}(s)\right)
,u_{\tau\varepsilon}(s)\right)  d\theta,\\
B_{\tau\varepsilon}^{i}(s)= & \int_{0}^{1}\sigma_{x}^{i}\left(  s,\bar
{X}(s)+\hat{X}(s)+\theta\left(  X_{\tau\varepsilon}^{u}(s)-X^{u}(s)-\hat
{X}(s)\right)  ,u_{\tau\varepsilon}(s)\right)  d\theta,\text{ }i=1,\ldots,d.
\end{align*}
Since $b_{xx}$ and $\sigma_{xx}$ are both Lipschitz continuous in $x$, from
(\ref{est-X1-X2}), one can verify that
\begin{equation}
\mathbb{E}\left[  \sup_{t\in\lbrack0,T]}\left\vert \int_{0}^{t}\Pi
_{\tau\varepsilon}(s)\mathrm{d}s\right\vert ^{2}+\sup_{t\in\lbrack0,T]}\left\vert
\sum\limits_{i=1}^{d}\int_{0}^{t}\Lambda_{\tau\varepsilon}^{i}(s)\mathrm{d}W_{i}%
(s)\right\vert ^{2}\right]  \leqslant C\varepsilon^{3}%
.\label{remainder-higher-order}%
\end{equation}
As $A_{\tau\varepsilon}(\cdot)$, $B_{\tau\varepsilon}^{i}(\cdot)$ are bounded,
(\ref{error-X-higher-order}) follows from (\ref{remainder-higher-order}) and a standard estimate for (\ref{eq-higher-order}).
	
Now we will accomplish the proof. Since $\Phi_{xx}$ and $f_{xx}$ are both Lipschitz continuous,
using (\ref{error-X-higher-order}) and (\ref{est-X1-X2}), we have
\begin{align}
\label{cost-func-expansion}
& J(u_{\tau\varepsilon}(\cdot))-J(u(\cdot))\\
\nonumber = & \ \mathbb{E}\left[  \Phi(X_{\tau\varepsilon}^{u}(T))-\Phi(X^{u}(T))+\int_{0}^{T}\left[  f(t,X_{\tau\varepsilon}^{u}(t),u_{\tau\varepsilon}(t))-f(t)\right]  \mathrm{d}t\right]  \\
\nonumber = & \ \mathbb{E}\left[  \Phi\left(  X^{u}+X_{1}+X_{2}\right)  (T)-\Phi(X^{u}(T))\right]  \\
\nonumber & +\mathbb{E}\left[  \int_{0}^{T}\left[  f(t,\left(  X^{u}+X_{1}+X_{2}\right)(t),u_{\tau\varepsilon}(t))-f(t, X^{u} (t),u_{\tau \varepsilon}(t))\right]  \mathrm{d}t\right]  \\
\nonumber & +\mathbb{E}\left[  \int_{0}^{T}\left[  f(t, X^{u} (t),u_{\tau \varepsilon}(t)) - f(t) \right]  \mathrm{d}t\right]  +R_{1}(\varepsilon)\\
\nonumber = & \ \mathbb{E}\left[  \Phi_{x}(X^{u}(T))\left(  X_{1}(T)+X_{2}(T)\right) + \frac{1}{2}\Phi_{xx}(X^{u}(T))X_{1}(T)X_{1}(T)\right]  \\
\nonumber & +\mathbb{E}\left[  \int_{0}^{T}\left[  f_{x}(t)\left(  X_{1}(t)+X_{2}(t)\right)  +\frac{1}{2}f_{xx}(t)X_{1}(t)X_{1}(t)\right]  \mathrm{d}t\right]  \\
\nonumber & +\mathbb{E}\left[  \int_{0}^{T}\left[  f(t,X^{u}(t),u_{\tau\varepsilon}(t))-f(t)\right]  \mathrm{d}t\right]  +R_{1}(\varepsilon)+R_{2}(\varepsilon)%
\end{align}
with
\begin{equation}
\label{est-remainder}
\left\vert R_{1}(\varepsilon)\right\vert \leqslant C \varepsilon^{\frac{3}{2}} , \quad \left\vert R_{2}(\varepsilon)\right\vert \leqslant C \varepsilon^{2},
\end{equation}
where $C>0$ is independent of $u(\cdot)$, $v(\cdot)$, $\tau$ and $\varepsilon$.
Employing (\ref{1st-adj-eq}), (\ref{2nd-adj-eq}), (\ref{cost-func-expansion}), and applying It\^{o}'s
lemma to $p^{u}(t)(X_{1}(t)+X_{2}(t))+\frac{1}{2}\mathrm{tr} \left\{ P^{u}(t)X_{1}(t)\left( X_{1}(t)\right)^{\intercal} \right\}$ on $[0,T]$ yields
\begin{align*}
  & J(u_{\tau\varepsilon}(\cdot))-J(u(\cdot))\\
= & \ \mathbb{E}\left[  \int_{0}^{T}\left[  H(t,X^{u}(t),p^{u}(t),q^{u}%
(t),u_{\tau\varepsilon}(t))-H(t,X^{u}(t),p^{u}(t),q^{u}(t),u(t))\right]
\mathrm{d}t\right]  \\
& \ +\frac{1}{2}\mathbb{E}\left[  \sum\limits_{i=1}^{d}\int_{0}^{T}\left(
\sigma^{i}(t,X^{u}(t),u_{\tau\varepsilon}(t))-\sigma^{i}(t)\right)
^{\intercal}P^{u}(t)\left(  \sigma^{i}(t,X^{u}(t),u_{\tau\varepsilon
}(t))-\sigma^{i}(t)\right)  \mathrm{d}t\right]  \\
& \ + R_{1}(\varepsilon)+R_{2}(\varepsilon)\\
= & \ \mathbb{E}\left[  \int_{E_{\tau\varepsilon}} \left[  \mathcal{H}(t,X^{u}(t),p^{u}%
(t),q^{u}(t),v(t),u(t))-\mathcal{H}(t,X^{u}(t),p^{u}%
(t),q^{u}(t),u(t),u(t))\right] \mathrm{d}t\right]
\\
& \ + R_{1}(\varepsilon)+R_{2}(\varepsilon)\\
= & \ \mathbb{E}\left[  \int_{E_{\tau\varepsilon}}\Delta_{u}\mathcal{H}%
(t)\mathrm{d}t\right] + R_{1}(\varepsilon)+R_{2}(\varepsilon).
\end{align*}
Combining this with (\ref{est-remainder}) implies (\ref{est-Jv-Ju}).
\end{proof}

\begin{remark}
Although the construction of $u_{\tau\varepsilon}(\cdot)$ in (\ref{def-spike-variation}) is a spike perturbation of the given $u(\cdot)$,
one may be unable to obtain (\ref{est-Jv-Ju}) by completely following the same approach as that to deriving the global stochastic maximum principle for (\ref{cost-func})-(\ref{state-eq}), 
due to the additional requirement that the constant $C$ must be uniform across $u(\cdot)$, $\tau$, and $\varepsilon$.
\end{remark}

Proposition \ref{prop-costfunc-dominate} provides us with a mechanism to reduce $J(u(\cdot))$ as long as $u(\cdot)$ violates the stochastic maximum principle.
To illustrate this mechanism, fix a $u(\cdot) \in \mathcal{U}[0,T]$ arbitrarily and put $\mathcal{T}_{u} = \left\{ t \in [0,T] : \mathbb{E} \left[ \Delta_{u} \mathcal{H}(t) \right] < 0 \right\}$. 
Only the following two situations occur:
\begin{itemize}
  \item $\mathrm{Leb}_{[0,T]}(\mathcal{T}_{u}) = 0$. Due to the definition of $\Delta_{u} \mathcal{H}$, one can easily deduce that $\Delta_{u}\mathcal{H}(t,\omega) = 0$, $\mathrm{Leb}_{[0,T]} \otimes \mathbb{P}$-a.e.,
which implies that $u(\cdot)$ satisfies the stochastic maximum principle (\ref{SMP}) and $J(u(\cdot))$ cannot be reduced by constructing a $u_{\tau \varepsilon}(\cdot)$.

  \item $\mathrm{Leb}_{[0,T]}(\mathcal{T}_{u}) > 0$. Employing the Lebesgue differentiation theorem, we have
        \[
        \lim\limits _{\varepsilon \rightarrow 0^{+}} \frac{1}{\varepsilon} \int_{E_{\tau \varepsilon}} \mathbb{E} \left[ \Delta_{u} \mathcal{H}(t) \right] \mathrm{d}t = \mathbb{E} \left[ \Delta_{u} \mathcal{H}(\tau) \right], \quad \mathrm{Leb}_{[0,T]} \text{-a.e. } \tau \in \mathcal{T}_{u},
        \]
        whence we deduce $\int_{E_{\tau \varepsilon}} \mathbb{E} \left[ \Delta_{u} \mathcal{H}(t) \right] \mathrm{d}t \leqslant -C_{u, \tau} \varepsilon < 0$ for sufficiently small $\varepsilon$ and some constant $C_{u, \tau}>0$ depending only on $u(\cdot)$ and $\tau$.
        Consequently, for $\mathrm{Leb}_{[0,T]}$-a.e. $\tau \in \mathcal{T}_{u}$ and sufficiently small $\varepsilon$, using Proposition \ref{prop-costfunc-dominate} yields
        \[
          J\left(  u_{\tau \varepsilon}(\cdot)\right) - J\left( u(\cdot) \right) \leqslant -C \varepsilon \left( \frac{C_{u,\tau}}{C} - \varepsilon^{\frac{1}{2}} \right) < 0.
        \]
\end{itemize}
To demonstrate the necessity of $P^{u}(\cdot)$, we need to compare it with the error estimates involving only the first-order adjoint processes.
Using our notation, for any $u(\cdot), u^{\prime}(\cdot) \in \mathcal{U}[0,T]$, the error estimate developed in \cite{BDL-MSA-2020}, Lemma 2.3 reads
\begin{align}
  \label{error-est-H}
  & J(u^{\prime}(\cdot)) - J(u(\cdot)) \\
\nonumber \leqslant & \ \mathbb{E}\left[ \int_{0}^{T} \left[ H(t, X^{u}(t), p^{u}(t), q^{u}(t), u^{\prime}(t)) - H(t, X^{u}(t), p^{u}(t), q^{u}(t), u(t)) \right] \mathrm{d}t \right] \\
\nonumber & \ + C \mathbb{E} \left[ \int_{0}^{T} \left\vert b(t, X^{u}(t), u^{\prime}(t)) - b(t, X^{u}(t), u(t)) \right\vert ^{2} \mathrm{d}t \right] \\
\nonumber & \ + C \mathbb{E} \left[ \int_{0}^{T} \left\vert \sigma(t, X^{u}(t), u^{\prime}(t)) - \sigma(t, X^{u}(t), u(t)) \right\vert ^{2} \mathrm{d}t \right] \\
\nonumber & \ + C \mathbb{E} \left[ \int_{0}^{T} \left\vert H_{x}(t, X^{u}(t), p^{u}(t), q^{u}(t), u^{\prime}(t)) - H_{x}(t, X^{u}(t), p^{u}(t), q^{u}(t), u(t)) \right\vert ^{2} \mathrm{d}t \right],
\end{align}
for some constant $C>0$ independent of $u(\cdot)$ and $u^{\prime}(\cdot)$.
Take $u(\cdot) = u_{\tau \varepsilon}(\cdot)$, which is defined by (\ref{def-spike-variation}), with $v(\cdot)$ being replaced by the following process $v^{\prime}(\cdot)$ determined by
\[
v^{\prime}(t,\omega) \in \underset{v \in U}{\arg\min} H(t, X^{u}(t,\omega), p^{u}(t,\omega), q^{u}(t,\omega), v), \quad (t,\omega) \in [0,T] \times \Omega.
\]
Then, using Lemma \ref{lem-est-state-eq} and Lemma \ref{adj-1st-lem}, we obtain
\begin{equation}
  \label{J-mechanism-H}
    J\left(  u_{\tau \varepsilon}(\cdot)\right) - J\left(  u(\cdot)\right)
\leqslant \mathbb{E}\left[  \int_{E_{\tau \varepsilon}}\Delta_{u} H(t)\mathrm{d}t \right] + C \left( \varepsilon + \mathbb{E}\left[  \int_{E_{\tau \varepsilon}} \left\vert q^{u}(t) \right\vert ^{2} \mathrm{d}t \right] \right)
\end{equation}
for some constant $C>0$ independent of $u(\cdot)$, $\tau$, and $\varepsilon$,
where 
$$
\Delta_{u} H(t) := H(t, X^{u}(t), p^{u}(t), q^{u}(t), v^{\prime}(t)) - H(t, X^{u}(t), p^{u}(t), q^{u}(t), u(t)), \quad t \in [0,T].
$$
For $\mathrm{Leb}_{[0,T]}$-a.e. $\tau \in \{ t \in [0,T]: \mathbb{E} \left[ \Delta_{u}H(t) \right] < 0\}$ and sufficiently small $\varepsilon$, 
following a similar analysis as above yields
\[
  J\left(  u_{\tau \varepsilon}(\cdot)\right) - J\left(  u(\cdot)\right)
\leqslant \left[ - C_{u, \tau} + C \left( 1 + C_{u, \tau}^{\prime}  \right) \right] \varepsilon
\]
with two positive constants $C_{u, \tau}$, $C_{u, \tau}^{\prime}$, depending only $u(\cdot)$ and $\tau$.
One cannot judge whether $J\left(  u_{\tau \varepsilon}(\cdot)\right)$ is less than $J\left(  u(\cdot)\right)$ since the sign of $C \left( 1 + C_{u, \tau}^{\prime}  \right) - C_{u, \tau}$ is unknown.
Unlike in (\ref{est-Jv-Ju}), the absence of $P^{u}(\cdot)$ in (\ref{J-mechanism-H}) prevents the construction of $u_{\tau \varepsilon}(\cdot)$ from neutralizing the influence of the positive remainder term in (\ref{error-est-H}).

\subsection{Construction and convergence of the algorithm}
Let $N$ be a positive integer. Put $\varepsilon_{N}^{}=T \cdot 2^{-N}$, $\tau_{j}^{N}=(2j-1)\varepsilon_{N}^{}$, $j=1,2,\ldots,2^{N-1}$.
We simply denote by $E_{j}^{N}$ the set $E_{\tau_{j}^{N} \varepsilon_{N}^{}}$ and by $u_{N_{j}}^{}(\cdot)$ the control $u_{\tau_{j}^{N} \varepsilon_{N}^{}}(\cdot)$. The following lemma describes a method to find such a $\tau$ mentioned above.

\begin{lemma}
\label{lem-tau-decrease}
Let Assumption \ref{assum-1} hold. Then, for any given integer $N \geqslant 1$, there exists at least one number $j \in \{1,2,\ldots,2^{N-1}\}$ such that
\begin{equation}
\label{ineq-tau-decrease-exist}
\mathbb{E}\left[  \int_{E_{j}^{N}}\Delta_{u}\mathcal{H}(t)\mathrm{d}t\right] \leqslant 2\varepsilon_{N}^{}\frac{\mu(u(\cdot))}{T},
\end{equation}
where $\mu(u(\cdot))$ is defined by (\ref{def-mu-control}).
\end{lemma}

\begin{proof}
The proof is almost same as Lemma 3.2 in \cite{BDL-MSA-2020} so we omit it.
\end{proof}

Based on Proposition \ref{prop-costfunc-dominate} and Lemma \ref{lem-tau-decrease}, we establish the following MSA algorithm to find the near-optimal control to (\ref{cost-func})-(\ref{state-eq}).
\begin{breakablealgorithm}
\caption{Algorithm of Successive Approximation for the Optimality of (\ref{cost-func})-(\ref{state-eq})}
\label{algorithm-MSA}
\begin{algorithmic}[1]
\State Let $u^{0}(\cdot) \in \mathcal{U}[0,T]$ be an initial guess.
\State Put $m=0$.
\State \textbf{Repeat}
\State For the given $u^{m}(\cdot)$, find $v^{m}(\cdot)$, $\Delta_{u^{m}}\mathcal{H}$ and $\mu(u^{m}(\cdot))$.
\State Put $N=1$.
\State For the given $N$, find the smallest $j \in \{1,2,\ldots,2^{N-1}\}$ such that (\ref{ineq-tau-decrease-exist}) holds.
\State Compute $J\left(u_{N_{j}}^{m}(\cdot)\right)$ for $\tau_{j}^{N}$ found from the step 6.
\If
{
\begin{equation}
\label{ineq-um-next-construct}
J\left(u_{N_{j}}^{m}(\cdot)\right)  -J\left(  u^{m}(\cdot)\right)\leqslant \varepsilon_{N}^{}\frac{\mu(u^{m}(\cdot))}{T}
\end{equation}
}
\State assign the values $u_{N_{j}}^{m}(t)$ to the control $u^{m+1}(t)$ for each $t \in [0,T]$; increase $m$ by unity: $m:=m+1$; proceed to the step 4;
\Else
\State proceed to the step 13.
\EndIf
\State Increase $N$ by unity: $N:=N+1$; proceed to the step 6.
\State \textbf{Until} $0 \leqslant J(u^{m}(\cdot)) - J(u^{m+1}(\cdot))$ is sufficiently small.
\State \textbf{Return} $u^{m}(\cdot)$.
\end{algorithmic}
\end{breakablealgorithm}

We have the following convergence result.

\begin{theorem}
\label{thm-MSA-convergence}
Let Assumption \ref{assum-1} hold. Then, for each integer $m\geqslant 1$, we have
\begin{equation}
\label{sequence-cost-decrease}
J\left(u^{m+1}(\cdot)\right)  -J\left(  u^{m}(\cdot)\right)\leqslant \frac{\mu^{3}(u^{m}(\cdot))}{2C^{2}T^{3}},
\end{equation}
where $C$ is the universal constant in (\ref{est-Jv-Ju}).
Moreover,
\begin{equation}
\label{mu-convergence-zero}
\lim\limits_{m\rightarrow\infty}\mu(u^{m}(\cdot))=0.
\end{equation}
\end{theorem}

\begin{proof}
Let $u^{m}(\cdot)$ be constructed and $N_{m}\geqslant 1$ be the minimal integer such that
\begin{equation}
\label{def-Nm}
\varepsilon_{N_{m}}^{} \leqslant \frac{\mu^{2}(u^{m}(\cdot))}{C^{2}T^{2}}.
\end{equation}
$N_{m}$ is the minimal number making (\ref{ineq-um-next-construct}) hold and we deduce from (\ref{est-Jv-Ju}) and (\ref{ineq-tau-decrease-exist}) that
\begin{equation}
J\left(  u_{\tau_{j}^{N_{m}}\varepsilon_{N_{m}}^{}}(\cdot)\right)  -J\left(
u^{m}(\cdot)\right)  \leqslant\left(  \frac{2\mu(u^{m}(\cdot))}{T}+C\sqrt
{\varepsilon_{N_{m}}^{}}\right)  \varepsilon_{N_{m}}^{}\leqslant\varepsilon_{N_{m}}^{}%
\frac{\mu(u^{m}(\cdot))}{T},
\label{minimal-integer-J-decrease}%
\end{equation}
which indicates that $u^{m+1}(\cdot)$ can be constructed by proceeding Algorithm \ref{algorithm-MSA} successfully.
Since $N_{m}$ is the minimal number in the series $1,2,\ldots$ for which (\ref{def-Nm}) holds, we further get
\begin{equation}
\label{Nm-minimal-property}
\varepsilon_{N_{m}}^{} > \frac{\mu^{2}(u^{m}(\cdot))}{2C^{2}T^{2}}.
\end{equation}
Noting that $\mu(u^{m}(\cdot))\leqslant 0$, combining (\ref{minimal-integer-J-decrease}) with (\ref{Nm-minimal-property}) yields (\ref{sequence-cost-decrease}).
Subsequently, we have
\[
-\mu^{3}(u^{m}(\cdot))\leq2C^{2}T^{3}\left[  J\left(  u^{m}(\cdot)\right)-J\left(  u^{m+1}(\cdot)\right)  \right]  .
\]
On adding these inequalities, we obtain
\[%
\begin{array}
[c]{rl}%
\sum\limits_{m=0}^{n-1}\left(  -\mu^{3}(u^{m}(\cdot))\right)  \leqslant &
2C^{2}T^{3}\left[  J\left(  u^{0}(\cdot)\right)  -J\left(  u^{n}%
(\cdot)\right)  \right]  \\
\leqslant & 2C^{2}T^{3}\left[  J\left(  u^{0}(\cdot)\right)  -\inf\limits_{u(\cdot
)\in\mathcal{U}[0,T]}J\left(  u(\cdot)\right)  \right]  \\
< & +\infty.
\end{array}
\]
We have $\sum\limits_{m=0}^{\infty}\left(  -\mu^{3}(u^{m}(\cdot))\right) < +\infty$ as $n \rightarrow \infty$,
which implies $\lim_{m\rightarrow\infty}\mu(u^{m}(\cdot))=0$.
\end{proof}

\subsection{Near-optimality and convergence rate}
In this section, we utilize Theorem \ref{thm-MSA-convergence} to approximate the global minimum of (\ref{cost-func})-(\ref{state-eq})
as far as possible and study the convergence rate in a specific case.

For each $m \in \mathbb{N}$, as mentioned earlier, $\mu(u^{m}(\cdot))$ characterizes the extent to which the resultant control $u^{m}(\cdot)$ produced by Algorithm \ref{algorithm-MSA}
deviates from satisfying the stochastic maximum principle (\ref{SMP}).
On the other hand, given $\delta>0$ small enough, from Theorem \ref{thm-MSA-convergence}, we deduce
\begin{equation}
\label{ineq-mu-m-delta}
-\delta \leqslant \mu(u^{m}(\cdot)) \leqslant 0
\end{equation}
for sufficiently large $m$, which means that $u^{m}(\cdot)$ satisfies (\ref{SMP}) approximately when $m$ is large enough. Generally, (\ref{ineq-mu-m-delta}) may not imply $J(u^{m}(\cdot)) - J(\overline{u}(\cdot)) \leqslant r(\delta)$ for sufficiently large $m$, 
where $r(\cdot)$ is a function of $\delta$ satisfying $r(\delta)\rightarrow 0$ as $\delta \rightarrow 0$. If $r(\delta)=\tilde{C}\delta^{\gamma}$ for some $\gamma>0$ independent of the constant $\tilde{C}$, then $u^{m}(\cdot)$ is called a near-optimal control with order $\delta^{\gamma}$. 
The following result shows that, under certain convex assumptions and for sufficiently large $m$, (\ref{ineq-mu-m-delta}) is sufficient
to make $r(\delta)=\tilde{C}\delta^{\frac{1}{2}}$ with a positive constant $\tilde{C}$ independent of $\delta$, which implies that $u^{m}(\cdot)$ is a near-optimal control with order $\delta^{\frac{1}{2}}$.

\begin{assumption}
\label{assum-2}
(i) $\Phi$ is convex in its argument.

(ii) $\psi$ is differentiable in $u$, and there exists a constant $\tilde{L}>0$ such that
\[
\left\vert \psi(t,x,u_{1})-\psi(t,x,u_{2}) \right\vert + \left\vert \psi_{u}(t,x,u_{1})-\psi_{u}(t,x,u_{2}) \right\vert \leqslant \tilde{L} \left\vert u_{1}- u_{2} \right\vert,
\]
where $\psi=b,\sigma,f$.
\end{assumption}

\begin{theorem}
\label{thm-delta-optimal}
Let Assumptions \ref{assum-1} and \ref{assum-2} hold, and $\delta>0$ be given. Then there exists a positive integer $N_{\delta}$, depending only on $\delta$, such that 
if $H(t,\cdot,p^{m}(t),q^{m}(t),\cdot)$ is convex for any $m \geqslant N_{\delta}$ and $\mathrm{Leb}_{[0,T]} \otimes \mathbb{P}$-a.e. $(t,\omega)$, then we have
\begin{equation}
\label{near-optimal-sqrt-delta}
J(u^{m}(\cdot))-J(\overline{u}(\cdot))\leqslant\tilde
{C}\delta^{\frac{1}{2}},
\end{equation}
where $\left( p^{m}(\cdot),q^{m}(\cdot) \right)$ is the solution to (\ref{1st-adj-eq}) corresponding to $\left( X^{m}(\cdot),u^{m}(\cdot) \right)$, and $\tilde{C}>0$ is a constant independent of $\delta$.
\end{theorem}

\begin{proof}
By Theorem \ref{thm-MSA-convergence}, for the given $\delta>0$, there exists a positive integer $N_{\delta}$ such that (\ref{ineq-mu-m-delta}) holds for each $m\geqslant N_{\delta}$. 
According to (\ref{def-mu-control}), one can rewrite (\ref{ineq-mu-m-delta}) as
\begin{align}
\label{ineq-suffi-1}
& \mathbb{E}\left[  \int_{0}^{T}\mathcal{H}(t,X^{m}(t),p^{m}(t),q^{m}(t),P^{m}(t),v^{m}(t),u^{m}(t))\mathrm{d}t\right]  \\
\nonumber & - \mathbb{E}\left[  \int_{0}^{T}\mathcal{H}(t,X^{m}(t),p^{m}(t),q^{m}(t),P^{m}(t),u^{m}(t),u^{m}(t))\mathrm{d}t\right] \geqslant -\delta.
\end{align}
Then, by (\ref{def-control-argmin-H}), one can verify
\begin{align}
\label{vm-inf-integral-H}
& \mathbb{E}\left[  \int_{0}^{T}\mathcal{H}(t,X^{m}(t),p^{m}(t),q^{m}(t),P^{m}(t),v^{m}(t),u^{m}(t))\mathrm{d}t\right]  \\
\nonumber = & \inf\limits_{u(\cdot)\in\mathcal{U}[0,T]}\mathbb{E}\left[  \int_{0}^{T}\mathcal{H}(t,X^{m}(t),p^{m}(t),q^{m}(t),P^{m}(t),u(t),u^{m}(t))\mathrm{d}t\right].
\end{align}
Subsequently, it follows from (\ref{ineq-suffi-1}) and (\ref{vm-inf-integral-H}) that
\begin{align}
\label{ineq-suffi-2}
& \mathbb{E}\left[  \int_{0}^{T}-\mathcal{H}(t,X^{m}(t),p^{m}(t),q^{m}(t),P^{m}(t),u^{m}(t),u^{m}(t))\mathrm{d}t\right]  \\
\nonumber & \geqslant \sup\limits_{u(\cdot)\in\mathcal{U}[0,T]}\mathbb{E}\left[  \int_{0}^{T}-\mathcal{H}(t,X^{m}(t),p^{m}(t),q^{m}(t),P^{m}(t),u(t),u^{m}(t))\mathrm{d}t\right]  -\delta.
\end{align}

If, for some $m \geqslant N_{\delta}$, $H(t,\cdot,p^{m}(t),q^{m}(t),\cdot)$ is convex for $\mathrm{Leb}_{[0,T]} \otimes \mathbb{P}$-a.e. $(t,\omega)$, 
then it is equivalent to the concavity of $-H(t,\cdot,p^{m}(t),q^{m}(t),\cdot)$ for $\mathrm{Leb}_{[0,T]} \otimes \mathbb{P}$-a.e. $(t,\omega)$. 
Under Assumption \ref{assum-2}, this and (\ref{ineq-suffi-2}) satisfy the conditions in Theorem 5.1 in \cite{Zhou-XY}. Consequently, there exists a constant $\tilde{C}>0$ independent of $\delta$ and $m$ such that
\begin{equation}
\label{near-optimal-sqrt-delta-inf}
J(u^{m}(\cdot))-\inf_{u(\cdot) \in \mathcal{U}[0,T]}J(u(\cdot))\leqslant\tilde
{C}\delta^{\frac{1}{2}}.
\end{equation}
Since we assume the existence of the optimal controls, we finally obtain (\ref{near-optimal-sqrt-delta}).
\end{proof}

\begin{remark}
Due to the definition of the $\mathcal{H}$-function (\ref{def-H-function}), one can rewrite it as
\[
\mathcal{H}(t,x,p,q,P,v,u)=H(t,x,p,q-P\sigma(t,x,u),v)+\dfrac{1}{2}%
\sum\limits_{i=1}^{d}\left(  \sigma^{i}(t,x,v)\right)  ^{\intercal}P\sigma^{i}(t,x,v),
\]
which is the form adopted by Zhou \cite{Zhou-XY}.
\end{remark}

\begin{remark}
If there is no optimal control to (\ref{cost-func})-(\ref{state-eq}), then one can only obtain (\ref{near-optimal-sqrt-delta-inf}) instead of (\ref{near-optimal-sqrt-delta}),
which implies that $u^{m}(\cdot)$ is a $\delta^{\frac{1}{2}}$-optimal control as long as $m \geqslant N_{\delta}$. Please refer to \cite{Zhou-XY} for the details about the near-optimal controls.
\end{remark}

\begin{corollary}
\label{cor-1}
Suppose Assumptions \ref{assum-1} and \ref{assum-2}. If $b$, $\sigma$ are linear functions with respect to $(x,u)$ and $f$ is convex in $(x,u)$ for each $t \in [0,T]$. 
Then there exist a positive integer $N_{\delta}$, depending only on $\delta$, such that (\ref{near-optimal-sqrt-delta}) holds for all $m \geqslant N_{\delta}$.
\end{corollary}

\begin{proof}
The proof is similar to that of Theorem \ref{thm-delta-optimal} as, for each $m \geqslant N_{\delta}$, the convexity of $H(t,\cdot,p^{m}(t),q^{m}(t),\cdot)$ holds naturally for $\mathrm{Leb}_{[0,T]} \otimes \mathbb{P}$-a.e. $(t,\omega)$.
\end{proof}

Now we provide a case where the convergence rate is available.
Let $b(t,x,u)=b_{1}(t)x+b_{2}(t)$, $\sigma(t,x,u)=\sigma(t,u)$, $\Phi
(x)=\frac{1}{2}x^{\intercal}\Gamma x$, $f(t,x,u)=\frac{1}{2}x^{\intercal
}G(t)x+g(t,u)$, where $\Gamma, G \in\mathbb{S}^{n\times n}$;
$b_{1}$ is a matrix-valued, bounded, deterministic process; $b_{2}$ is an $n$-dimensional,
vector-valued, bounded, deterministic process; $\sigma:[0,T]\times U\longmapsto\mathbb{R}^{n\times d}$; $g:[0,T]\times U\longmapsto\mathbb{R}$.

\begin{theorem}
\label{thm-convergence-rate}
Let Assumption \ref{assum-1} hold and $b$, $\sigma$, $\Phi$, $f$ be defined as above. If $\overline{u}(\cdot) \in \mathcal{U}[0,T]$ is an optimal control to (\ref{cost-func})-(\ref{state-eq}), then we have
\begin{equation}
\label{ineq-rate}
0 \leqslant J(u^{m}(\cdot))-J(\overline{u}(\cdot)) \leqslant \tilde{C} m^{-\frac{1}{2}}, \ \ m \in \mathbb{N}_{+},
\end{equation}
where the sequence $\left\{ u^{m}(\cdot) \right\}_{m}$ is produced by Algorithm \ref{algorithm-MSA}, and
$$
\tilde{C}=\max \left\{ J(u^{1}(\cdot)) - J(\overline{u}(\cdot)), 2C^{-2}T^{-3} \right\}
$$ 
with $C>0$ being the universal constant appearing in Theorem \ref{thm-MSA-convergence}.
\end{theorem}

To prove Theorem \ref{thm-convergence-rate}, we need the following proposition.

\begin{proposition}
\label{prop-sequence-sqrt-order}
Let $\{a_{m}\}_{m\in \mathbb{N}_{+}}$ be the sequence of nonnegative numbers such that
\begin{equation}\label{sequence-difference-3-power}
a_{m+1}-a_{m} \leqslant -Aa_{m}^{3},
\end{equation}
where the constant $A>0$ is given. Then $a_{m}=O(m^{-\frac{1}{2}})$.
\end{proposition}

\begin{proof}
Let $a_{m}=b_{m}\cdot m^{-\frac{1}{2}}$ for some nonnegative sequence $\{b_{m}\}_{m\in \mathbb{N}_{+}}$.
Then it is enough to show that $b_{m}$ is bounded for all $m \in \mathbb{N}_{+}$. Through (\ref{sequence-difference-3-power}), we obtain
\[
a_{m}-a_{m+1}=\frac{b_{m}}{\sqrt{m}}\left(  1-\frac{b_{m+1}}{b_{m}}
\sqrt{\frac{m}{m+1}}\right)  \geqslant A\left(  \frac{b_{m}}{\sqrt{m}}\right)^{3},
\]
whence we can deduce
\[
1-\frac{b_{m+1}}{b_{m}}\sqrt{\frac{m}{m+1}}\geqslant A\frac{b_{m}^{2}}{m}.
\]
After some transformation, we can rewrite the inequality above as
\[
\frac{b_{m+1}}{b_{m}}\leqslant\sqrt{1+\frac{1}{m}}\left(  1-A\frac{b_{m}^{2}}%
{m}\right)  .
\]
Thus, we have
\[
\frac{b_{m+1}}{b_{m}}\leqslant\left(  1+\frac{1}{m}\right)  \left(  1-A\frac
{b_{m}^{2}}{m}\right)  =1+\frac{1}{m}\left(  1-Ab_{m}^{2}\right)
-A\frac{b_{m}^{2}}{m^{2}}.
\]
If $1-Ab_{m}^{2}<0$, we have
\[
\frac{b_{m+1}}{b_{m}}\leq1+\frac{1}{m}\left(  1-Ab_{m}^{2}\right)
-A\frac{b_{m}^{2}}{m^{2}}<1.
\]
Hence $b_{m+1}<b_{m}$. Otherwise, we have $b_{m} \leqslant A^{-\frac{1}{2}}$.
Consequently, we conclude that $b_{m} \leqslant \max \{b_{1}, A^{-\frac{1}{2}} \}$ for all $m \in \mathbb{N}_{+}$.
The proof is complete.
\end{proof}

\begin{proof}[Proof of Theorem \ref{thm-convergence-rate}]
Subtracting $J(\overline{u}(\cdot))$ from $J(u^{m}(\cdot))$ yields
\begin{align}
\label{Jum-Jbaru-1}
& J(u^{m}(\cdot))-J(\overline{u}(\cdot))\\
\nonumber = & \ \mathbb{E}\left[  \Phi(X^{m}(T))-\Phi(\overline{X}(T))+\int_{0}^{T}\left[ f(t,X^{m}(t),u^{m}(t))-f(t,\overline{X}(t),\overline{u}(t))\right]  \mathrm{d}t\right]  \\
\nonumber = & \ \mathbb{E}\left[  \left(  X^{m}(T)\right)  ^{\intercal}\Gamma(X^{m}(T)-\overline{X}(T))-\frac{1}{2}\left(  X^{m}(T)-\overline{X}(T)\right)^{\intercal}\Gamma(X^{m}(T)-\overline{X}(T))\right.  \\
\nonumber & + \int_{0}^{T}\left[  \left(  X^{m}(t)\right)  ^{\intercal}G(t)(X^{m}(t)-\overline{X}(t))-\frac{1}{2}\left(  X^{m}(t)-\overline{X}(t)\right)  ^{\intercal}G(t)(X^{m}(t)-\overline{X}(t))\right]  \mathrm{d}t\\
\nonumber & + \left.  \int_{0}^{T}\left[  g(t,u^{m}(t))-g(t,\overline{u}(t))\right] \mathrm{d}t\right].
\end{align}
Observe that, for any $u(\cdot)\in \mathcal{U}[0,T]$, (\ref{1st-adj-eq}) becomes
\[
p^{u}(t)=\Gamma X^{u}(T)+\int_{t}^{T}\left(  b_{1}^{\intercal}(s)p^{u}%
(s)+G(s)X^{u}(s)\right) \mathrm{d}s-\sum_{i=1}^{d}\int_{t}^{T}q^{u,i}(s)\mathrm{d}W^{i}%
(s),\text{ \ }t\in\lbrack0,T],
\]
and (\ref{2nd-adj-eq}) becomes
\[
P(t)=\Gamma+\int_{t}^{T}\left[  b_{1}^{\intercal}(s)P(s)+P^{\intercal}%
(s)b_{1}(s)+G(s)\right] \mathrm{d}s,\text{ \ }t\in\lbrack0,T].
\]
Then, applying It\^{o}'s lemma yields
\begin{align}
\label{Jum-Jbaru-2}
0 \leqslant & J(u^{m}(\cdot))-J(\overline{u}(\cdot)) \\
\nonumber = & \ \mathbb{E}\left[  \int_{0}^{T}\left[  \left\langle q^{m}(t),\sigma(t,u^{m}(t))-\sigma(t,\overline{u}(t))\right\rangle +g(t,u^{m}(t))-g(t,\overline{u}(t))\right]  \mathrm{d}t\right.  \\
\nonumber & -\left.  \frac{1}{2}\sum\limits_{i=1}^{d}\int_{0}^{T}\left(  \sigma^{i}(t,u^{m}(t))-\sigma^{i}(t,\overline{u}(t))\right)  ^{\intercal}P(t)(\sigma^{i}(t,u^{m}(t))-\sigma^{i}(t,\overline{u}(t)))\mathrm{d}t\right]  \\
\nonumber = & \ \mathbb{E}\left[  \int_{0}^{T}\left[  \mathcal{H} (t,X^{m}(t),p^{m}(t),q^{m}(t),P(t),u^{m}(t),u^{m}(t)) \right. \right. \\
\nonumber & - \left. \left. \mathcal{H}(t,X^{m}(t),p^{m}(t),q^{m}(t),P(t),\overline{u}(t),u^{m}(t)) \right]  \mathrm{d}t \right]  \\
\nonumber \leqslant & \ \mathbb{E}\left[  \int_{0}^{T}\left[  \mathcal{H}(t,X^{m}(t),p^{m}(t),q^{m}(t),P(t),u^{m}(t),u^{m}(t)) \right. \right. \\
\nonumber & - \left. \left. \mathcal{H}(t,X^{m}(t),p^{m}(t),q^{m}(t),P(t),v^{m}(t),u^{m}(t)) \right]  \mathrm{d}t \right]  \\
\nonumber = & \ -\mu(u^{m}(\cdot)).
\end{align}
For each $m \in \mathbb{N}_{+}$, define the nonnegative sequence by $a_{m}=J(u^{m}(\cdot))-J(\overline{u}(\cdot))$. Then, as $-\mu(u^{m}(\cdot)) \geqslant 0$, it follows from (\ref{sequence-cost-decrease}) and (\ref{Jum-Jbaru-2}) that
\[
a_{m+1}-a_{m}
= \left[  J(u^{m+1}(\cdot))-J(\overline{u}(\cdot))\right]  -\left[  J(u^{m}%
(\cdot))-J(\overline{u}(\cdot))\right]
\leqslant -\frac{a_{m}^{3}}{2C^{2}T^{3}},
\]
which verifies Proposition \ref{prop-sequence-sqrt-order} with $A=2C^{-2}T^{-3}$.
Hence we have 
$$
a_{m} \leqslant \max \left\{ a_{1}, 2C^{-2}T^{-3} \right\} m^{-\frac{1}{2}}, \quad \forall m \in \mathbb{N}_{+}.
$$
The proof is complete.
\end{proof}

The following simple example shows that, when $U$ is non-convex, the modified MSA algorithms --- although very powerful 
--- are rather sensitive to the choice of the modification parameter $\rho \geqslant 0$ appearing in (\ref{intro-extended-H-mini}). 
(For more details about modified MSA algorithms, for instance, please refer to \cite{BDL-MSA-2020}, Algorithm 1, or, \cite{Sethi-Siska-MSA}, Algorithm 1.)
Compared with the modified MSA algorithms, the introduction of the second-order adjoint equation (\ref{2nd-adj-eq}) in our Algorithm \ref{algorithm-MSA} yields more robust performance.

\begin{example}
  \label{sec-3-example}
  Let $n=d=k=T=1$, $x_{0} = 0$, $G=\Gamma=2$, $b_{1} = b_{2} = 0$, $\sigma(t,u) = u$, $g(t,u) = -\frac{1}{2}u^{2}$, and $U = \{ 0,1 \}$.
By simple calculation, it can be verified that 
\[
J(u(\cdot)) = \mathbb{E} \left[ \int_{0}^{1} \left( \frac{3}{2} - t \right) u(t)^{2} \mathrm{d}t \right], \quad \forall u(\cdot) \in \mathcal{U}[0,1],
\]
and hence $\overline{u}(\cdot) \equiv 0$ is the unique optimal control (see \cite{YongZhou}, Example 3.1).
The usual Hamiltonian is given by $H(t,x,p,q,v) = qv + x^{2} - v^{2}/2$ and then we can write
\[
\mathcal{H}(t,x,p,q,P,v,u) = qv + x^{2} - \frac{1}{2} v^{2} + \frac{1}{2} P v(v - 2u).
\]
For any $u(\cdot) \in \mathcal{U}[0,1]$, $\left( p^{u}(\cdot), q^{u}(\cdot) \right)$ satisfies
\[
p^{u}(t)= 2 X^{u}(1) + 2 \int_{t}^{1}  X^{u}(s) \mathrm{d}s - \int_{t}^{1}q^{u}(s)\mathrm{d}W(s), \quad t \in [0,1],
\]
and $P(t) = 4 - 2t$, $t \in [0,1]$. Using It\^{o}'s formula to $P(t)X^{u}(t)$ on $t \in [0,1]$, one can demonstrate that $p^{u}(t) = P(t)X^{u}(t)$, $q^{u}(t) = P(t)u(t)$,
and then we obtain
\begin{equation}
  \label{example-H-function}
\mathcal{H}(t,X^{u}(t),p^{u}(t),q^{u}(t),P(t),v,u(t)) = \left( \frac{3}{2} - t \right)v^{2} + \left\vert X^{u}(t) \right\vert ^{2},
\quad \forall (t,v) \in [0,1] \times U.
\end{equation}
For any $t \in [0,1]$, it is easy to check that $\Delta_{u} \mathcal{H}(t) = (3/2 - t) [v(t)^{2} - u(t)^{2}] < 0$ if and only if $u(t) = 1$, $v(t) = 0$.
Thus, the set $\mathcal{T}_{u} = \{ t \in [0,1]: \mathbb{E} [ \Delta_{u} \mathcal{H}(t) ] < 0 \} = \{ t \in [0,1]: \exists \omega \in \Omega, \ u(t,\omega) = 1 \}$
and then $u(t) = 1_{\mathcal{T}_{u}}(t)$, $t \in [0,1]$.

Choose $u^{0}(\cdot) \in \mathcal{U}[0,1]$ arbitrarily such that $\mathrm{Leb}_{[0,1]}(\mathcal{T}_{u^{0}}) > 0$.
On the one hand, according to Algorithm \ref{algorithm-MSA}, we have $v^{0}(\cdot) = 0$, $\tau_{1}^{1} = 1/2$, $\varepsilon_{1} = 1/2$, $E_{\tau_{1}^{1} \varepsilon_{1}} = [0,1]$, $u_{\tau_{1}^{1} \varepsilon_{1}}^{0}(\cdot) = v^{0}(\cdot)$, and 
\[
\mu(u^{0}(\cdot)) = \mathbb{E} \left[ \int_{0}^{1} \Delta_{u^{0}} \mathcal{H}(t) \mathrm{d}t \right] = - \int_{\mathcal{T}_{u}} \left( \frac{3}{2} - t \right) \mathrm{d}t = - J(u^{0}(\cdot)).
\]
(\ref{ineq-tau-decrease-exist}) holds naturally and (\ref{ineq-um-next-construct}) is verified as
\[
J(u_{\tau_{1}^{1} \varepsilon_{1}}^{0}(\cdot)) - J(u^{0}(\cdot)) = - J(u^{0}(\cdot)) = \mu(u^{0}(\cdot)) < \frac{\mu(u^{0}(\cdot))}{2}.
\]
Due to Step 9 in Algorithm \ref{algorithm-MSA}, we assign $u_{\tau_{1}^{1} \varepsilon_{1}}^{0}(\cdot)$ to $u^{1}(\cdot)$ and then $u^{1}(\cdot) = 0 = \overline{u}(\cdot)$, which means that we attain the global minimum after one iteration.

On the other hand, thanks to the identity $q^{u}(t) = P(t)u(t)$, the augmented Hamiltonian (\ref{intro-extended-H-mini}) now reads
\begin{equation}
\label{example-aug-Hamilton}
\left\{
\begin{array}[c]{lr}
  \frac{\rho - 1}{2} \left[ v + \frac{P(t) - \rho}{\rho - 1} u(t) \right]^{2} - \frac{1}{2} \left[ \frac{ \left( P(t) - \rho \right)^{2} }{\rho - 1} - \rho \right] u(t)^{2} + X^{u}(t) ^{2}, & \rho \neq 1, \\
  \ & \ \\
  \left[ P(t) - 1 \right] u(t)v + X^{u}(t) ^{2} + \frac{\rho}{2} u(t)^{2}, & \rho = 1.
\end{array}
\right.
\end{equation}
With the same initial guess $u^{0}(\cdot)$ as above, the efficiency of the modified MSA depends on the choice of $\rho \geqslant 0$, which is demonstrated as follows.
\begin{itemize}
  \item $0 \leqslant \rho < 1$. We have $v^{0}(t) = 1_{[0,1] \setminus \mathcal{T}_{u^{0}}}(t)$, $t \in [0,1]$. The sign of
  \[
  J(v^{0}(\cdot)) - J(u^{0}(\cdot)) = 1 - 2 J(u^{0}(\cdot))
  \]
  relies on the value of $J(u^{0}(\cdot))$, so the modified MSA is inapplicable to this case.

  \vspace{0.1cm}
  
  \item $1 \leqslant \rho < 3$. We have $v^{0}(\cdot) = 0 = \overline{u}(\cdot)$, whence we assign $v^{0}(\cdot)$ to $u^{1}(\cdot)$.
  The global minimum is attained after one iteration.
  
  \vspace{0.1cm}

  \item $3 \leqslant \rho < 7$. We have $v^{m}(t) = 1_{\mathcal{T}_{u^{m}} \cap [(7 - \rho)/4,1]} (t)$ for each $t \in [0,1]$ and any $m \in \mathbb{N}$,
  and assign $v^{m}(\cdot)$ to $u^{m+1}(\cdot)$ as long as 
  \[
  0 \leqslant J(u^{m}(\cdot)) - J(v^{m}(\cdot)) = \int_{\mathcal{T}_{u^{m}} \setminus \mathcal{T}_{u^{m+1}}} \left( \frac{3}{2} - t \right) \mathrm{d}t < \delta
  \]
  with the set $\mathcal{T}_{u^{m+1}} := \mathcal{T}_{u^{m}} \cap [(7 - \rho)/4,1]$ and any given permissible threshold $\delta > 0$.  
  
  \vspace{0.1cm}

  \item $\rho \geqslant 7$. We have $v^{0}(\cdot) = u^{0}(\cdot)$. In this case, we cannot use the modified MSA to reduce the value of $J$ once $u^{0}(\cdot)$ is given.
  Therefore, $u^{0}(\cdot)$ can be viewed as a local minimum over the admissible controls belonging to $\{ u(\cdot) \in \mathcal{U}[0,1]: J(u(\cdot)) \geqslant J(u^{0}(\cdot)) \}$,
  which is nonempty since it contains the admissible control $u(\cdot) \equiv 1$.
\end{itemize}
It is observed that an excessively small $\rho$ can render the modified MSA inoperative, whereas an excessively large $\rho$ may drive the modified MSA into a local minimum.

\end{example}

\bibliographystyle{MSA_arXiv.bst}
\bibliography{MSA_arXiv.bib}

\end{document}